\documentclass[12pt]{article}

\usepackage{amsmath}
\usepackage{amsfonts}
\usepackage{amssymb, amsfonts, amsmath, nicefrac}
\usepackage{color}
\usepackage{amsthm}
\usepackage{graphicx}
\usepackage{makecell}
\usepackage{diagbox}
\usepackage{subcaption}
\usepackage[ruled,vlined]{algorithm2e}

\newcommand{\cM}{{\mathfrak M}}

\newcommand{\cR}{{\mathfrak R}}

\newcommand{\cp}{{\mathfrak p}}

\newcommand{\iid}{\stackrel{\mathrm{iid}}{\sim}}

\newtheorem{theorem}{Theorem}[section]
\newtheorem{lemma}{Lemma}[section]
\newtheorem{proposition}{Proposition}[section]

\newtheorem{remark}{Remark}[section]

\newtheorem{assu}{Assumption}[section]

\newcommand{\R}{{\cal R}}

\def\argmin{\mathop{\rm argmin}}
\newcommand{\var}{{\rm Var}}

\newcommand{\V}{{\cal V}}

\newcommand{\B}{{\cal B}}

\newcommand{\M}{{\cal M}}
\newcommand{\X}{{\cal X}}

\newcommand{\W}{\mathcal{W}}
\newcommand{\s}{{\cal S}}

\newcommand{\half}{ \mbox{\small$\frac{1}{2}$}}

\newcommand{\be}{\begin{equation}}
\newcommand{\ee}{\end{equation}}

\newcommand{\bxi}{\mbox{\small\boldmath$\xi$}}

\def\dist{\mathop{\rm dist}}

\def\e{\epsilon}

\def\vv{\vartheta}

\def\bbr{{\Bbb{R}}} 
\def\bbe{{\Bbb{E}}} 

\def\bbd{{\Bbb{D}}}

\begin{document}

\title{\bf Bayesian Distributionally Robust Optimization}

\author{\bf Alexander Shapiro, Enlu Zhou, Yifan Lin \\
School of Industrial and Systems Engineering\\
Georgia Institute of Technology}

\maketitle
\thispagestyle{empty}

\noindent
{\bf Abstract.} We introduce a new framework, Bayesian Distributionally Robust Optimization (Bayesian-DRO), for data-driven  stochastic optimization where the underlying distribution is unknown. Bayesian-DRO contrasts with most of the existing DRO approaches in the use of Bayesian estimation of  the unknown distribution. To make computation of Bayesian updating tractable, Bayesian-DRO first assumes the underlying distribution takes a parametric form with unknown parameter and then computes the posterior distribution of the parameter. To address the model uncertainty brought by the assumed parametric distribution, Bayesian-DRO constructs an ambiguity set of distributions with the assumed parametric distribution as the reference distribution and then optimizes with respect to the worst case in the ambiguity set. We show the consistency of the Bayesian posterior distribution and subsequently the convergence of objective functions and optimal solutions of Bayesian-DRO.   Our consistency result of the Bayesian posterior requires simpler assumptions than the classical literature on Bayesian consistency. We also consider several approaches for selecting the ambiguity set size in Bayesian-DRO and compare them numerically. Our numerical experiments demonstrate the out-of-sample performance of Bayesian-DRO in comparison with Kullback-Leibler-based (KL-) and Wasserstein-based empirical DRO as well as risk-neutral Bayesian Risk Optimization. Our numerical results shed light on how to choose the modeling framework (Bayesian-DRO, KL-DRO, Wasserstein-DRO) for specific problems, but the choice for general problems still remains an important and open question.

\setcounter{equation}{0}

\section{Introduction}\label{sec-intr}
Consider the following  stochastic optimization problem
\begin{equation}\label{stat-1}
\min_{x\in \X} \bbe_Q[G(x,\xi)],
\end{equation}
where $\X\subset \bbr^n$ is a nonempty closed  set, $Q$ is a probability distribution of random vector $\xi$   supported on $\Xi\subset \bbr^d$, and $G:\X\times \Xi\to\bbr$ is the cost function. The notation
\begin{equation}\label{integr}
    \bbe_Q[Z]=\int_\Xi Z(\xi)dQ(\xi)
\end{equation}
emphasizes that the expectation is taken with respect to the probability measure\footnote{Probability measure $Q$ is defined on the sample (measurable) space $(\Xi,\B)$, where $\B$ is the Borel sigma algebra of $\Xi$.} (distribution) $Q$ of random variable (measurable function) $Z:\Xi\to \bbr$.
We use the same notation $\xi$ viewed as random vector or as its realization, the particular meaning will be clear from the context.

In many applications, the underlying `true' distribution of $\xi$ is not known and should be derived (estimated) from the available data.  A popular approach to deal with this distributional uncertainty is to construct  an ambiguity set $\cM$ of probability   distributions and to consider  the following minimax (worst-case) counterpart of problem \eqref{stat-1}:
\begin{equation} \label{distr-1}
\min_{x \in \X}\sup_{Q\in \cM} \bbe_{Q}[G(x, \xi)].
\end{equation}

Such Distributionally Robust Optimization (DRO) approach to stochastic programming has a long history. In the setting of an inventory model, it was considered in the pioneering paper \cite{scarf1958min}. Various methods have been developed for construction of the ambiguity sets, such as methods based on moment constraints (e.g., \cite{delage2010distributionally}), $\phi$-divergence (e.g. \cite{bental1987penalty}), Wasserstein distance (e.g., \cite{Esfahani-Kuhn}), and Bayesian guarantees \cite{gupta2019near}.

A different approach is to fit a parametric family $P_\theta$, $\theta\in \Theta$,  of distributions to the (observed)  data $(\xi_1,...,\xi_N)$. We assume that the parameter set $\Theta\subset \bbr^k$ is closed, and that the parametric family is defined  by density $f(\cdot|\theta)$.
The value of the parameter vector $\theta$  is then  estimated, say by the Maximum Likelihood method. This involves two approximations of the `true' distribution. First, the parametric family is just a model, and as the famous quote is saying ``every model is wrong, but some are useful". Second, the estimated value of the parameter vector may be not accurate especially when the available data are limited. The popular Bayesian approach is aimed at reducing variability of the parameter evaluation. That is, the parameter vector $\theta$  is assumed to be random whose probability distribution is supported on the set $\Theta$ and defined by a prior probability density $p(\theta)$. Then given the data (sample) $\bxi^{(N)}=(\xi_1,...,\xi_N)$, the posterior distribution is determined by Bayes' rule
\begin{equation}\label{stat-2}
 p(\theta|\bxi^{(N)})=\frac{f(\bxi^{(N)}|\theta) p(\theta)}{\int_\Theta f(\bxi^{(N)}|\theta)p(\theta) d\theta},
\end{equation}
where  $f(\bxi^{(N)}|\theta)=\prod_{i=1}^N f(\xi_i|\theta)$ is the conditional density of the sample by assuming  $\xi_i$'s are independent and identically distributed (i.i.d.).

Recently, \cite{wu2018bayesian} takes the Bayesian approach with the motivation to use the Bayesian posterior distribution (which encodes the likelihoods of all possibilities) to replace the ambiguity set (which treats every possibility inside the set with equal probability), and further take a risk functional with respect to the posterior distribution to allow more flexible risk attitude. This leads to the following  Bayesian Risk Optimization (BRO) formulation
\begin{equation}\label{BRO}
\min_{x \in \X}\rho_{\theta_N}\left( \bbe_{\xi|\theta}[G(x, \xi)]\right),
\end{equation}
where $\rho_{\theta_N}$ is a risk functional (such as expectation, mean-variance, Value-at-Risk, Conditional Value-at-Risk) taken with respect to the posterior distribution $p(\theta|\bxi^{(N)})$, and $\bbe_{\xi|\theta}$ is the expectation taken with respect to the parametric distribution $f(\xi|\theta)$ conditional on $\theta$. However, as mentioned above, the assumed parametric family introduces model uncertainty.

In this paper, we propose a new formulation termed Bayesian Distributionally Robust Optimization (Bayesian-DRO), which poses robustness against
the model uncertainty (ambiguity)  of the assumed parametric distributions while maintaining the advantage of Bayesian estimation when data are limited. It constructs an ambiguity set by taking the parametric distribution as the reference distribution and optimizes the worst-case of the Bayesian average of the true problem.  More specifically, for every $\theta\in \Theta$  let $\cM^\theta$ be a set of probability measures on $(\Xi,\B)$.
We propose the following DRO formulation: 

\begin{equation}\label{stat-5a}
\min_{x \in \X}\bbe_{\theta_N}\left[\sup_{Q\in\cM^\theta}  \bbe_{Q|\theta}[G(x,\xi)]\right],
\end{equation}
where  $\bbe_{Q|\theta}$ is the expectation with respect to distribution $Q$ of $\xi$ conditional on $\theta$
and
\begin{equation}\label{postexp}
\bbe_{\theta_N}[Y]:=\int_\Theta  Y(\theta)p(\theta|\bxi^{(N)})d\theta
\end{equation}
denotes the expectation of random variable $Y:\Theta\to\bbr$  with respect to the posterior distribution $p(\theta|\bxi^{(N)})$.
We refer to $\cM^\theta$ as the {\em ambiguity set}; a specific construction of the ambiguity sets will be discussed in the next section.   Please note that the posterior distribution depends on choice of the prior density  $p(\theta)$ and parametric model  $f(\cdot|\theta)$. The choice of $p(\theta)$ and $f(\cdot|\theta)$   could both be subject to ambiguity.  In this paper we mainly deal with ambiguity with respect to the parametric model. In Section~\ref{sec-variants} we give a brief discussion of modeling ambiguity of the posterior distribution, which of course also depends on ambiguity of the  prior density.

We show the consistency of Bayesian posterior distributions. In particular, when the parametric model is mis-specified (i.e., when the true distribution lies outside the parametric family of distributions), the posterior distribution converges to the parametric distribution which has the minimum Kullback-Leibler (KL) divergence (within the parametric family) from the true distribution.  Consistency of Bayesian posterior distribution under model mis-specification has been studied in the literature (e.g.,  \cite{Ghosal2000ConvergenceRO,Kleijn2006MisspecificationII, Lian2009Posterior}), but the assumptions required in our results are in general simpler and easier to verify than constructing a testing sequence as usually required in the existing literature. Built on this result, we show the objective functions and optimal solutions of Bayesian-DRO are strongly consistent.

When the ambiguity set is constructed using the KL divergence and its radius is small, we show that Bayesian-DRO is approximately equivalent to a weighted sum of the mean and standard deviation under the posterior distribution, where the weight depends on the size of the ambiguity set. This reveals that the robustness of Bayesian-DRO comes from the trade-off between the posterior mean and variability of the solution performance.   Similar interpretation of robustness has also been observed in divergence-based empirical DRO (see \cite{duchi2021statistics, gotoh2018robust}), but the difference is that  empirical DRO trades off the empirical mean and standard deviation (i.e., with respect to the empirical distribution) and in Bayesian-DRO these are with respect to the posterior distribution.  To determine the ambiguity set size, we propose several theoretical and empirical methods and compare their performance numerically.

The rest of the paper is organized as follows. Section~\ref{sec-distr} formally introduces the Bayesian-DRO formulation, discusses the construction of the ambiguity set, and understands the robustness of Bayesian-DRO by sensitivity analysis. Section~\ref{sec-conv} analyzes convergence of Bayesian-DRO and considers how to determine the size of the ambiguity set. Section~\ref{sec-numeric} presents numerical results to illustrate the performance of Bayesian-DRO in comparison with empirical DRO as well as BRO-mean. Section~\ref{sec-conclusion} concludes the paper with a brief discussion of future work.

\setcounter{equation}{0}
\section{Bayesian distributionally robust optimization}
\label{sec-distr}
The risk neutral formulation of the Bayesian counterpart of problem \eqref{stat-1} can be written as
\begin{equation}\label{stat-3}
\min_{x\in \X} \left\{g(x):=\bbe_{\theta_N}\left[ \bbe_{\xi|\theta}[G(x,\xi)]\right]\right\},
\end{equation}
where the expectation $\bbe_{\xi|\theta}$ is taken with respect to the distribution of $\xi$ conditional on $\theta$, defined by density $f(\cdot|\theta)$, and the expectation  $\bbe_{\theta_N}$ is taken with respect to the posterior distribution  $p(\theta|\bxi^{(N)})$ defined in \eqref{stat-2}. Note that the nested expectation in \eqref{stat-3} can be considered as the expectation with respect to the joint  distribution of $\xi$ and $\theta$. An unbiased estimate of $g(x)$ can be obtained by generating a random realization of $\theta$ according to the posterior distribution $p(\theta|\bxi^{(N)})$  and then generating a random realization of $\xi\sim f(\cdot|\theta)$ conditional on generated $\theta$. This allows to apply either the Sample Average Approximation (SAA) or Stochastic Approximation (SA) optimization methods for solving problem  \eqref{stat-3},  provided that there is an efficient way  to generate such random samples. 

Now let us consider the uncertainty  with respect to the choice of the parametric family of distributions of $\xi$, with a specified prior distribution of $\theta$,   which is often taken as an uninformative prior when there is no prior knowledge. We view \eqref{stat-3} as the {\em nominal model } with observed (given)  data $\bxi^{(N)}$, and the reference parametric family defined by the probability density function (pdf) $f(\cdot|\theta)$, $\theta\in \Theta$. We assume that the ambiguity set $\cM^\theta$ consists of probability measures defined by density functions, i.e., every distribution of the ambiguity set has respective pdf $q(\cdot|\theta)$, $\theta\in \Theta$. We also assume that the ambiguity set contains the nominal distribution. There are many ways how the ambiguity set can be constructed, and we will discuss a specific construction, well suited for our purposes, in Section \ref{sec-constr} below.

  By constructing an ambiguity $\cM^\theta$ for each fixed $\theta$ in \eqref{stat-3}, we define the Bayesian distributionally robust optimization problem \eqref{stat-5a}, which is re-stated below for clarity:
\begin{equation}\label{BDRO}
\min_{x \in \X}\bbe_{\theta_N}\left[\sup_{Q\in\cM^\theta}  \bbe_{Q|\theta}[G(x,\xi)]\right].
\end{equation}

For this problem,  define the following distributionally robust functional
\begin{equation}\label{stat-5}
\cR(Z) :=\bbe_{\theta_N}\left[\sup_{Q\in\cM^\theta}  \bbe_{Q|\theta}[Z]\right].
\end{equation}

This  functional is defined on an appropriate  linear space of measurable functions (random  variables) $Z:\Xi\to\bbr$. The functional $\cR$ can be viewed as a nested conditional functional. We can refer to \cite{pichler:2020} for a detailed discussion  of such conditional functionals. For random variable $Z:\Xi\to\bbr$, the respective expectation  in  \eqref{stat-5} is
\begin{equation}\label{denf}
\bbe_{Q|\theta}[Z]=\int_\Xi Z(\xi)q(\xi|\theta) d\xi,
\end{equation}
where $q(\cdot|\theta)$ is the pdf of $Q\in \cM^\theta$. The maximum (supremum) in the right hand side of  \eqref{stat-5} is taken over all pdfs $q_\theta(\xi)=q(\xi|\theta)$ from the ambiguity set $\cM^\theta$.

The distributionally robust  counterpart of problem \eqref{stat-3} is  obtained by employing     the above distributionally robust functional. That is, problem  \eqref{BDRO}  can be written as
\begin{equation}\label{stat-6}
\min_{x\in \X} \cR(G_x),
\end{equation}
where $G_x(\xi):=G(x,\xi)$.
Of course, it should be verified that the above distributionally robust functionals are well defined for every $Z=G_x$, $x\in \X$. We will discuss this in the next section.

\begin{remark} \rm
In problem \eqref{BDRO}, if we take $\cM^{\theta}$ constant for all $\theta\in\Theta$, i.e., $\cM^\theta = \cM, \forall \theta\in\Theta$, then this problem becomes a DRO problem
\begin{equation}\label{eqn1}
\min_{x\in\mathcal{X}} \sup_{Q\in\cM}\bbe_Q[G(x,\xi)].
\end{equation}

Hence, $\{\cM^\theta, \theta\in\Theta\}$ in \eqref{BDRO} can be viewed as a finer characterization of the ambiguity set based on the likelihood of each $\theta$, whereas the usual DRO takes a ``blanket'' ambiguity set for every $\theta$. Moreover, the outer expectation in \eqref{BDRO} aggregates all $\theta\in\Theta$ by their posterior density rather than fixating on the worst case in the ambiguity set.

When $\cM^\theta$ is a singleton consisting of only $f(\cdot|\theta)$, then \eqref{BDRO} reduces to  \eqref{stat-3} or BRO-mean (i.e., \eqref{BRO} with expectation being the risk functional). This implies that as opposed to BRO-mean, Bayesian-DRO imposes additional robustness with respect to the possibly misspecified likelihood.
\end{remark}


\subsection{Construction of the ambiguity set}
\label{sec-constr}

Consider now construction of the ambiguity set for the parametric family.
The functional
\begin{equation}\label{risk-8}
 \varrho_{|\theta} (\cdot):=\sup_{Q\in\cM^\theta}  \bbe_{Q|\theta}[\,\cdot\,]
\end{equation}
can be viewed as a coherent risk measure conditional on $\theta\in \Theta$.  We have that $\bbe_{Q|\theta}[Z]$ is a function of $\theta\in \Theta$  defined by the corresponding integral (see \eqref{denf}) which is assumed to be well defined. It could happen that by taking the maximum (supremum) of such functions over possibly uncountable family of distributions, the resulting value    $\varrho_{|\theta} (Z)$, considered as a function of $\theta\in \Theta$, is not measurable. In that case the corresponding integral, defining $\cR(Z)$, does not exist. We will deal with this issue in the specific construction below.

There are many ways how the ambiguity sets can be constructed. The following approach, of the so-called $\phi$-divergence (\cite{csiszar1964informationstheoretische},\cite{morimoto1963markov}), is general and flexible. Let $\phi:\bbr\to \bbr_+\cup\{+\infty\}$ be a convex lower semi-continuous function such that $\phi(1)=0$ and $\phi(x)=+\infty$ for $x<0$. For $\e\ge 0$  and  $f_\theta(\xi):= f(\xi|\theta)$ define the corresponding set of pdfs $q_\theta(\xi)=q(\xi|\theta)$, representing the ambiguity set, as

\begin{equation}\label{densit}
\cM^\theta_\e:=\left\{q_\theta:\int_\Xi\phi\big( q_\theta(\xi)/f_\theta(\xi) \big) f_\theta(\xi)d\xi\le \e\right\}.
\end{equation}

That is, the ambiguity set consists of pdfs having $\phi$-divergence $\le \e$ from  the reference parametric pdf  $f(\xi|\theta)$. Note that $\cM^\theta_\e$ contains the reference measure (distribution) defined by the pdf $f(\xi|\theta)$. Note also that the probability measure defined by the pdf $q_\theta$ in \eqref{densit} is assumed to be absolutely continuous with respect to the reference measure $f_{\theta}$ for every $\theta\in \Theta$.

Consider  the conjugate $\phi^*(y)=\sup_{x\ge 0}\{yx-\phi(x)\}$ of $\phi$. Note that  the conjugate of $\lambda \phi(\cdot)$  is  $(\lambda \phi)^*(y)=\lambda \phi^*(y/\lambda)$ for $\lambda>0$.
It can be shown by duality arguments (cf., \cite{bayraksan2015data},\cite{bental1987penalty},\cite{shapiro2017distributionally}), that for a random variable $Z:\Xi\to \bbr$,
\begin{equation}\label{dualphi}
   \varrho_{|\theta}(Z)  =\inf_{\lambda \ge 0,\mu}\left\{
 \lambda \e+\mu+\bbe_{\xi|\theta}\big[(\lambda \phi)^*(Z-\mu)\big]\right\}.
\end{equation}
Hence, the functional \eqref{stat-5} can be  written as
\begin{equation}
 \label{funr-2}
 \cR(Z) = \bbe_{\theta}\Big[\underbrace{\inf_{\lambda> 0,\mu} \bbe_{\xi|\theta} \big [ \lambda\e+\mu+\lambda \phi^*\big((Z-\mu)/\lambda\big) \big ]}_{ \varrho_{|\theta}(Z)}\Big].
 \end{equation}

The measurability of the infimum $ \varrho_{|\theta}(Z)$  in the right hand side of \eqref{dualphi}, considered as a function of $\theta$, can be verified under mild regularity conditions. For example, we have the following result.

\begin{proposition}
Suppose that for almost every (with respect to the Lebesgue measure) $\xi$ the density  function $f(\xi|\theta)$ is lower semicontinuous  in $\theta\in \Theta$. Then $ \varrho_{|\theta}(Z)$ is measurable in $\theta$.
\end{proposition}

\begin{proof}
Since the conjugate function $\phi^*(\cdot)$ is lower semicontinuous and $f(\xi|\cdot)$ is lower semicontinuous,   we have that for almost every $\xi$ the function
 $\lambda\e+\mu+\lambda \phi^*\big((Z(\xi)-\mu)/\lambda\big) f(\xi|\theta)$ is
   lower semicontinuous  in $(\lambda,\mu,\theta)$. It follows by Fatou's lemma that its integral
   \[
   \bbe_{\xi|\theta} \big [ \lambda\e+\mu+\lambda \phi^*\big((Z-\mu)/\lambda\big) \big ]=\int_\Xi [ \lambda\e+\mu+\lambda \phi^*\big((Z(\xi)-\mu)/\lambda\big)   ]
f(\xi|\theta) d\xi
\]
is  lower semicontinuous  in $(\lambda,\mu,\theta)$ and hence is measurable.
Therefore  the above integral is a normal integrand  \cite[Corollary 14.41]{rockafellar2009variational}, and hence its  infimum over $(\lambda,\mu)\in \bbr_+\times \bbr$ is measurable   \cite[Theorem 14.37]{rockafellar2009variational}.
\end{proof}

\subsubsection{Kullback-Leibler divergence}
The Kullback-Leibler (KL) divergence from  a pdf $q(\cdot)$ to a pdf $f(\cdot)$, on $\Xi$, is
\begin{equation}\label{kldiver}
 D_{KL}(q\|f):=\int_\Xi q(\xi)\ln \big (q(\xi)/f(\xi)\big) d\xi=
 \int_\Xi (q(\xi)/f(\xi))\ln \big(q(\xi)/f(\xi)\big)  f(\xi) d\xi.
\end{equation}

The  KL-divergence is a particular instance of the
$\phi$-divergence with
\[
\phi(x):=x\ln x-x+1,\;x\ge 0.
\]

The corresponding ambiguity set $\cM_\e^\theta$ is formed by pdfs $q_\theta$ such that $D_{KL}(q_\theta\|f_\theta)\le \e$. We will show that the KL-divergence approach is in accordance with the consistency of the Bayesian posterior distribution in Section~\ref{sec-Bayesian}, and therefore is a natural approach to construction of the corresponding ambiguity set.

We make the following assumption in the remainder of the paper:
for $x\in \X$ and $Z:=G_x$ it follows that
\begin{equation}\label{assum}
 \bbe_{\xi|\theta} \big[ e^{
t Z}\big]<+\infty\;\;\text{ for any}\;   t\in \bbr\;\text{and}\;\theta\in \Theta.
\end{equation}

For the KL-divergence, given $\lambda>0$ the minimizer over $\mu$ in \eqref{dualphi} is given by
$
\mu=\lambda  \ln \bbe_{\xi|\theta} \big[e^{Z/\lambda}\big],
$
and hence the minimum becomes
\begin{equation}\label{kld-2}
  \varrho_{|\theta}(Z) = \inf_{\lambda > 0}\left\{
 \lambda \e  +\lambda\ln \bbe_{\xi|\theta} \big[ e^{
 Z/\lambda}\big]\right\}.
\end{equation}

Consequently, the DRO problem \eqref{stat-5} can be written as
\begin{equation}\label{dro-nest}
\min_{x\in \X}\bbe_{\theta_N}\left[\inf_{\lambda>0}\big\{\lambda\e+\lambda\ln \bbe_{\xi|\theta}  [ e^{
 G_x/\lambda}  ]\big\}\right].
\end{equation}

The above optimization problem \eqref{dro-nest} can be viewed as a two-stage stochastic  program with the second stage given by the optimization problem with respect to $\lambda>0$. It can be solved, for example, by the Sample Average Approximation (SAA) method; we will discuss this further in Section \ref{sec-numeric}.

\subsubsection{Robustness via sensitivity analysis}
We now consider the sensitivity of the Bayesian-DRO objective value with respect to $\e$, size of the ambiguity set. Note that for $\e=0$ the minimum (infimum)  in \eqref{kld-2} is attained  as $\lambda\to+\infty$, and equals $\bbe_{\xi|\theta}[Z]$. For $\e>0$ the optimization problem \eqref{kld-2} has unique optimal solution $\bar{\lambda}$, with $\bar{\lambda}$ tending to $+\infty$ as $\e\downarrow 0$.

Consider  minimization problem in the right hand side of  \eqref{kld-2} for $\theta\in \Theta$ and small $\e>0$. By  condition \eqref{assum},  the log-moment generation function $\Lambda(t):=\ln\bbe_{\xi|\theta}\left[e^{tZ}\right]$ is finite valued and  infinitely differentiable with its  first and second  derivatives at $t=0$ being the respective mean and variance. Consequently, by using the second order Taylor expansion of the log-moment generating function, we can write
\begin{equation}\label{appr-1}
 \lambda \e  +\lambda\ln \bbe_{\xi|\theta} \big[ e^{
 Z/\lambda}\big]=  \lambda \e +\mu +\half  \sigma^2/\lambda +O(\lambda^{-2}),
\end{equation}
where\footnote{Of course, $\mu$ and $\sigma$   depend on $\theta$, we suppress this in the notation.}
$\mu:=\bbe_{\xi|\theta}[Z]$ , $\sigma^2:=\var_{\xi|\theta}(Z)$ by minimizing the right hand side of \eqref{appr-1} we obtain approximation $ \bar{\lambda} \approx \frac{ \sigma}{ \sqrt{2\e}}$ of the optimal solution of \eqref{kld-2}, and consequently for small $\e>0$ the approximation
\begin{equation}\label{appr-2}
 \min_{\lambda > 0}\left\{
 \lambda \e  +\lambda\ln \bbe_{\xi|\theta} \big[ e^{Z/\lambda}\big]\right\}
 \approx \mu+ \sigma \sqrt{2\e}.
\end{equation}

Plugging the approximation \eqref{appr-2} into the Bayesian-DRO problem \eqref{dro-nest} reveals that when the ambiguity set is small, Bayesian-DRO is approximately equal to a weighted sum of the posterior mean and posterior standard deviation of the performance function, with weight depending on the ambiguity set size $\e$. A similar interpretation of mean-variance trade-off has also been observed for divergence-based empirical DRO (see \cite{duchi2021statistics, gotoh2018robust}), but its mean and standard deviation are with respect to the empirical distribution. Moreover, \cite{gotoh2021calibration} shows that the empirical DRO can be interpreted as a trade-off between the mean and worst-case sensitivity (we refer the reader to \cite{gotoh2021calibration} for the definition of worst-case sensitivity); whether such an interpretation can be extended to Bayesian-DRO will be left as a future work.


\subsubsection{Variants of Bayesian-DRO Formulations}\label{sec-variants}

In this section we briefly discuss some other possible DRO formulations in the Bayesian setting and their pros and cons. We first consider the alternative formulation
 \begin{equation}\label{dro-1}
\min_{x\in \X, \lambda>0}\bbe_{\theta_N}\left[\lambda\e+\lambda\ln \bbe_{\xi|\theta}  [ \exp(
 G_x/\lambda)  ]\right].
\end{equation}
The nested Bayesian-DRO problem \eqref{dro-nest} can be viewed as a relaxation of problem \eqref{dro-1}. In \eqref{dro-1} the parameter $\lambda$ is chosen before observing  a realization of $\theta$, while in \eqref{dro-nest} the parameter $\lambda$ is a function of $\theta$. We have that the optimal value of the Bayesian-DRO  problem \eqref{dro-nest} is less than or equal to the optimal value of problem \eqref{dro-1}. It could be noted that the relaxation   \eqref{dro-1} is computationally easier to solve than  \eqref{dro-nest}, since it avoids nested Monte Carlo simulation that is needed in solving the nested formulation \eqref{dro-nest}.

Now let's consider another variant of formulation.
  As it was mentioned in section \ref{sec-intr}, the posterior distribution depends on the choice of the prior density and parametric family.
In the above derivations we considered the ambiguity with respect to the reference parametric  pdf $f(\cdot|\theta)$,  and consequently  the corresponding Bayesian-DRO problem \eqref{BDRO}. It is possible to apply the KL-divergence  ambiguity  approach to the posterior distribution rather than the parametric family. That is for $\e>0$  let $\M_\e$ be the family of pdfs $\cp(\theta)$, $\theta\in \Theta$,  such that $D_{KL}\left(\cp\|p(\cdot|\bxi^{(N)})\right)\le \e$. Let
\begin{equation}\label{dklfunc}
 \R(Y):=\sup_{\cp\in \M_\e}\left\{\bbe_\cp[Y]= \int_\Theta Y(\theta) \cp(\theta) d\theta\right\}
\end{equation}
be the corresponding distributionally robust functional defined on a space of random variables $Y:\Theta\to \bbr$. Similar to \eqref{kld-2} we have the following representation of that functional
\begin{equation}\label{dklfun-2}
 \R(Y)=\inf_{\lambda > 0}\left\{
 \lambda \e  +\lambda\ln \bbe_{\theta_N} \big[ e^{
 Y/\lambda}\big]\right\}.
\end{equation}

The corresponding DRO problem is obtained by replacing the expectation $\bbe_{\theta_N}$ in \eqref{stat-3} with $\R$, that is  minimization of $\R\big(\bbe_{\xi|\theta}[G_x]\big)$ over $x\in \X$. By \eqref{dklfun-2} we can write this optimization problem as
 \begin{equation}\label{dklfun-3}
\min_{x\in \X, \lambda>0}
  \lambda\e+\lambda\ln \bbe_{\theta_N}   \big[ \exp\big(\bbe_{\xi|\theta}[ G_x/\lambda]\big)\big  ] .
\end{equation}

Now  by interchanging the expectation $\bbe_{\theta_N}$ and the supremum in the definition \eqref{stat-5} of the distributionally robust functional $\cR$, we can consider the functional
\begin{equation}\label{distfn-1}
\Re(Z) :=\sup_{ Q\in\cM^\theta}\bbe_{\theta_N}\left[  \bbe_{Q|\theta}[Z]\right]
\end{equation}
and the corresponding Bayesian-DRO problem. We have that $\Re(\cdot)\le \cR(\cdot)$ and the inequality can be strict since the extreme measure $Q$ in  \eqref{stat-5} could depend on $\theta$. The maximization in \eqref{distfn-1}  is over the pdfs of the ambiguity set. Because the expectation with respect to these pdfs is inside the expectation $\bbe_{\theta_N}$, it is not clear how to represent the corresponding optimization problem in the KL-divergence framework. It is also not clear what could  be  an interpretation of the functional $\Re$  and the corresponding optimization problem.

\setcounter{equation}{0}
\section{Analysis}
\label{sec-conv}
Suppose that the data  $\xi_1,\ldots,\xi_N$  are generated i.i.d. from the true (data-generating) distribution $Q_*$, i.e., $\xi_i\iid Q_*$, and that $Q_*$ has density (pdf) denoted $q_*$.
Recall that $p(\theta)$ denotes the prior pdf, $f(\xi|\theta)$ denotes the reference parametric family, and  $p(\theta|\bxi^{(N)})$ denotes the posterior pdf as defined in \eqref{stat-2}.

\subsection{Consistency of Bayesian posterior distributions}\label{sec-Bayesian}
In this section we discuss  convergence of the posterior pdf $p(\theta|\bxi^{(N)})$ as $N$ goes to infinity.  The analysis of this section is a first step in establishing consistency of the Bayesian-DRO, discussed in the next section. We   make the following assumptions which are relatively easy to verify and well suited for the considered framework.

\begin{assu} \label{assump}
{\rm (i)} The set $\Theta$ is convex  compact with nonempty interior.
{\rm (ii)} $\ln p(\theta)$ is bounded on $\Theta$, i.e., there are constants $c_1>c_2>0$ such that $c_1\ge p(\theta)\ge c_2$ for all $\theta\in \Theta$.
{\rm (iii)}  $q_*(\xi)>0$  for $\xi\in \Xi$.
{\rm (iv)} $f(\xi|\theta)>0$, and hence  $p(\theta|\bxi^{(N)})>0$,  for all $\xi\in \Xi$ and $\theta\in \Theta$. {\rm (v)}     $f(\xi|\theta)$ is continuous in $\theta\in \Theta$.
{\rm (vi)} $\ln f(\xi|\theta)$, $\theta\in \Theta$,  is dominated by an integrable (with respect to $Q_*$)  function.
\end{assu}
Assumptions~\ref{assump}(i)-(ii)   provide    sufficient conditions for uniform convergence of the posterior distribution. Without these assumptions, convergence of the posterior still holds but may not be uniform. The rest of Assumption~\ref{assump} are regularity assumptions.

Consider function
\begin{equation}\label{funpsi}
 \psi(\theta):=\bbe_{q_*}\big [\ln f(\xi|\theta)\big]=\int_\Xi   \ln  f(\xi|\theta) Q_*(d \xi)=
 \int_\Xi q_*(\xi) \ln  f(\xi|\theta)d\xi.
\end{equation}

Under Assumption~\ref{assump},  the function $\psi:\Theta\to \bbr$ is real valued. Moreover,    we have that for $\theta\in \Theta$,
\[
\lim_{\theta'\to\theta}\psi(\theta') = \lim_{\theta'\to\theta} \int_\Xi  \ln  f(\xi|\theta')  Q_*(d \xi)=
\int_\Xi  \lim_{\theta'\to\theta} \ln  f(\xi|\theta')  Q_*(d \xi)=\psi(\theta),
\]
where we use continuity of $f(\xi|\theta)$ in $\theta$, and  the interchange of  the limit and integral follows by the  Dominated Convergence Theorem since $\ln f(\cdot|\theta)$ is dominated by an integrable function. Thus   $\psi(\theta)$ is continuous on $\Theta$.

Consider the KL-divergence
\[
D_{KL}\big(q_*\| f_\theta\big)=\int_\Xi q_*(\xi)\ln\left(\frac{q_*(\xi)}{f(\xi|\theta)}\right )d\xi=\bbe_{q_*}[\ln q_*(\xi)]-\underbrace{\bbe_{q_*}[\ln f(\xi|\theta)]}_{\psi(\theta)}.
 \]

Let
\[
\Theta^*:=\arg\min_{\theta\in \Theta} D_{KL}(q_*\| f_\theta)= \arg\max_{\theta\in \Theta} \underbrace{\bbe_{q_*}[\ln f(\xi|\theta)]}_{\psi(\theta)}.
\]

Since the set $\Theta$ is compact and $\psi(\cdot)$ is continuous, it follows that the set $\Theta^*$ is nonempty. Note that if the model is correct, then $\Theta^*=\{\theta\in \Theta:q_*=f_\theta\}$.

For a point $\theta^*\in \Theta^*$ and $\e>0$, define the sets
\begin{equation}\label{vsets}
V_\e:=\{\theta\in \Theta:\psi(\theta^*)-\psi(\theta)\ge  \e\},\;
 U_\e:=\Theta\setminus  V_\e=
 \{\theta\in \Theta:\psi(\theta^*)-\psi(\theta)<  \e\}.
\end{equation}

Since $\psi(\theta^*)=\max_{\theta\in \Theta}\psi(\theta)$, the sets $V_\e$  and $U_\e$  remain the same for any $\theta^*\in \Theta^*$.
Note that $U_\e$ is a  neighborhood of the set $\Theta^*$. Since the set $\Theta$ is convex with nonempty interior,  it follows that  volume $\int_{U_\e}d\theta$, of the set $U_\e$, is greater than zero for any $\e>0$.

The following theorem shows that the posterior pdf $p(\theta|\bxi^{(N)})$ converges almost surely to a distribution with probability mass   concentrated on $\Theta^*$. If $\Theta^*$ is the singleton $\{\theta^*\}$, then $p(\theta|\bxi^{(N)})$ converges almost surely to the Dirac delta function $\delta(\theta^*)$. The convergence is uniform in $\theta\in\Theta$ regardless of the choice of the prior pdf $p(\theta)$. In what follows by writing w.p.1 (almost surely) we mean that the considered property holds with probability one with respect to the probability measure $Q_*^\infty$. Construction  of the  probability measure $Q_*^\infty$ for the sequence $\{\xi_1,...\}$ is verified  by Kolmogorov's existence theorem. By saying that: ``a property holds w.p.1 for $N$ large enough",  we mean that there is a subset of the considered probability space having  measure zero such that for any element of the probability space outside this measure-zero set, there is $N'$ (depending on that element) such that the property holds for that element for any $N\ge N'$.


\begin{lemma} Suppose that  Assumption~\ref{assump} holds.
 Then for    $0<\beta<\alpha <\e$, it follows that  w.p.1 for $N$ large enough
\begin{equation}
\label{convest}
\sup_{\theta\in V_\e}  p(\theta|\bxi^{(N)})\le \kappa(\beta)^{-1} e^{-N(\alpha-\beta)} ,
\end{equation}
where  $V_\e$ and $U_\e$  are  defined in \eqref{vsets},  and\,\footnote{Recall that under Assumption \ref{assump},   $\kappa(\beta) >0$ for any $\beta>0$.} $\kappa(\beta):=\int_{U_\beta} d\theta$.
\end{lemma}

\begin{proof}
Define
\[
\phi_N(\theta):=N^{-1}\ln f(\bxi^{(N)}|\theta)=N^{-1}\sum_{i=1}^N \ln f(\xi_i|\theta).
\]
By the  Law of Large Number (LLN) we have for $\theta\in \Theta$ that
\begin{equation}\label{conv-1}
\lim_{N\to\infty}  \phi_N(\theta)=\psi(\theta), \; \rm w.p.1.
\end{equation}
Hence we can write
\begin{equation}\label{conv-2}
  N^{-1}\ln [f(\bxi^{(N)}|\theta) ]=\psi(\theta)+\varepsilon_N(\theta),
\end{equation}
where $\varepsilon_N(\theta)$  tends to 0 w.p.1 for any $\theta\in \Theta$.
Now for $\theta^*\in \Theta^*$ and $\theta\in V_\e$ we have
\begin{equation}\label{equat-1}
\ln p(\theta^*|\bxi^{(N)})-\ln p(\theta|\bxi^{(N)})=
\ln f(\bxi^{(N)}|\theta^*)-\ln f(\bxi^{(N)}|\theta)+\ln p(\theta^*)-\ln p(\theta).
\end{equation}
It  follows by \eqref{conv-2} that
\begin{equation}\label{ineq-1}
N^{-1}[ \ln p(\theta^*|\bxi^{(N)})-\ln p(\theta|\bxi^{(N)})]= \psi(\theta^*)-\psi(\theta) +
\varepsilon_N(\theta^*)-\varepsilon_N(\theta)+N^{-1}[\ln p(\theta^*)- \ln  p(\theta)].
\end{equation}
Consider a point $\theta\in V_\e$. Then
\begin{equation}\label{ineq-2}
N^{-1}[ \ln p(\theta^*|\bxi^{(N)})-\ln p(\theta|\bxi^{(N)})]\ge  \e+\gamma_N (\theta),
\end{equation}
where $\gamma_N (\theta)$ tends to zero w.p.1.
It follows  that for any $\alpha\in (0,\e)$,   w.p.1 for $N$ large enough
\begin{equation}\label{ineq-3}
 \ln p(\theta^*|\bxi^{(N)})-\ln p(\theta|\bxi^{(N)})\ge N \alpha,
\end{equation}
or equivalently
\begin{equation}\label{ineq-4}
e^{-N\alpha} p(\theta^*|\bxi^{(N)})\ge p(\theta|\bxi^{(N)}).
\end{equation}
In the similar way by using \eqref{ineq-1}, we obtain for $\theta\in U_{\e}$ and $\beta\in (0,\e)$ that  w.p.1 for $N$ large enough
\[
 \ln p(\theta^*|\bxi^{(N)})-\ln p(\theta|\bxi^{(N)})\le N \beta,
 \]
or equivalently
\begin{equation}\label{ineq-4a}
e^{-N\beta} p(\theta^*|\bxi^{(N)})\le p(\theta|\bxi^{(N)}).
\end{equation}
Now let us show that w.p.1 for $N$ large enough
\begin{equation}\label{bound-1}
 p(\theta^*|\bxi^{(N)})\le e^{N\beta}/\kappa(\beta).
\end{equation}
Indeed
    since $p(\theta|\bxi^{(N)})$ is a density we have
\[
1=\int_\Theta p(\theta|\bxi^{(N)}) d\theta\ge \int_{U_\beta} p(\theta|\bxi^{(N)}) d\theta
\ge e^{-N\beta}\kappa(\beta)p(\theta^*|\bxi^{(N)}),
\]
where for the last inequality  we used \eqref{ineq-4a} with $\kappa(\beta)=\int_{U_\beta} d\theta$.

By Assumption~\ref{assump} the set $\Theta$ is compact and $\ln f(\xi|\theta)$, $\theta\in \Theta$,  is dominated by an integrable (with respect to $Q_*$)  function.   Then by the uniform LLN (e.g., \cite[Theorem 7.48]{shapiro2021lectures})  the limit \eqref{conv-1} can be strengthened to the uniform limit
\begin{equation}\label{conu-1a}
\lim_{N\to\infty}\sup_{\theta\in \Theta} \left|  \phi_N(\theta)-\psi(\theta)\right|=0, \;{\rm w.p.1},
\end{equation}
i.e., $\varepsilon_N(\theta)=N^{-1}\ln [f(\bxi^{(N)}|\theta) ]-\psi(\theta)$ tends to 0 w.p.1 uniformly in $\theta\in \Theta$.
Assumption~\ref{assump} further supposes that
 $\ln p(\theta)$ is bounded on $\Theta$, i.e., there are constants $c_1>c_2>0$ such that $c_1\ge p(\theta)\ge c_2$ for all $\theta\in \Theta$.    Then
  \begin{equation}\label{ineq-1a}
N^{-1}[ \ln p(\theta^*|\bxi^{(N)})-\ln p(\theta|\bxi^{(N)})]= \psi(\theta^*)-\psi(\theta) +
 \eta_N(\theta),
\end{equation}
 where  $$\eta_N(\theta):=\varepsilon_N(\theta^*)-\varepsilon_N(\theta)+N^{-1}[\ln p(\theta^*)- \ln  p(\theta)]$$
  tends  to 0 w.p.1 uniformly in $\theta\in \Theta$.
 Thus for any $\alpha\in (0,\e)$ we have that w.p.1 for $N$ large enough
   \begin{equation}\label{ineq-1b}
 \ln p(\theta^*|\bxi^{(N)}) \ge N\alpha+ \sup_{\theta\in V_\e}\ln p(\theta|\bxi^{(N)}).
\end{equation}
By \eqref{bound-1} it follows that
   for $0<\beta<\alpha <\e$, w.p.1 for $N$ large enough
\begin{equation}\label{conu-2}
\sup_{\theta\in V_\e}  p(\theta|\bxi^{(N)})\le e^{-N\alpha} p(\theta^*|\bxi^{(N)})\le e^{-N(\alpha-\beta)} /\kappa(\beta).
\end{equation}
This completes the proof.
\end{proof}

Let $\theta_N$ be random vector with  the  posterior  pdf $p(\theta|\bxi^{(N)})$. We have that probability of the event $\{\theta_N\in V_\e\}$ is given by the integral $\int_{V_\e} p(\theta|\bxi^{(N)}) d\theta$.
Consequently  under Assumption \ref{assump},  we have by \eqref{convest} that for any $\e>0$, w.p.1 for $N$ large enough,
\begin{equation}\label{ineq-6}
{\rm Prob}\{\theta_N\in V_\e\}  \le    \kappa(\beta)^{-1} \nu e^{-N(\alpha-\beta)},
\end{equation}
where $\nu$ is   volume of the set $\Theta$.
It follows that   probability of the event $\{\theta_N\in U_\e\}$
converges w.p.1 to one as $N\to \infty$.
Note that for an appropriate $\e>0$,  the set $U_\e=\Theta\setminus V_\e$ can be an arbitrarily tight neighborhood of the set $\Theta^*$. Therefore by \eqref{ineq-6} we have the following result.

\begin{theorem}
Suppose that  Assumption~{\rm \ref{assump}} holds. Then with
 w.p.1 the distance from  $\hat{\theta}_N$  to the set $\Theta^*$  converges in probability to zero. In particular if $\Theta^*=\{\theta^*\}$ is the singleton, then for almost every sequence $\{\xi_1,...\}$,  we have that $\theta_N$ converges in probability to $\theta^*$.
 \end{theorem}


\begin{remark}
{\rm
Convergence of Bayesian posterior distributions has been studied for a long time, dating back to Doob's consistency \cite{doob1948application}. We refer the reader to \cite{Ghosal2000ConvergenceRO} for a nice overview of Bayesian consistency results. Our analysis here resembles the proof and result of Schwartz consistency \cite{Schwartz1965OnBP}, but we do not require the assumption of the existence of a testing sequence, which is a common assumption in many of Bayesian consistency results (e.g., \cite{Schwartz1965OnBP,Ghosal2000ConvergenceRO,Kleijn2006MisspecificationII,Gao2020Posterior}) but usually hard to verify in practice. Instead we impose simpler and maybe stronger assumptions (see Assumption~\ref{assump}). These assumptions are easy to verify and sufficient for our problems.}
\end{remark}

\subsection{Consistency of Bayesian optimization problems}
As in the previous section by writing w.p.1   we mean this with respect to the probability measure $Q_*^\infty$.
Consider a function $H:\X\times \Theta\to \bbr$ and the corresponding optimization problem
\begin{equation}\label{optimiz-1}
  \min_{x\in \X} \left\{ \bbe_{\theta_N} [H_x]=\int_\Theta H(x,\theta)
  p(\theta|\bxi^{(N)})d\theta\right\}.
\end{equation}

In this section we discuss convergence of the optimal value and the set of optimal solutions of the above problem as $N\to\infty$. In the considered applications the function $H(x,\theta)$ is given by
\begin{equation}\label{funch}
\begin{array}{l}
H(x,\theta):=\bbe_{\xi|\theta}[G(x,\xi)]\;\;{\rm and}\;\;
H(x,\theta):= \sup_{Q\in\cM^\theta}  \bbe_{Q|\theta}[G(x,\xi)]
\end{array}
\end{equation}
in the cases of the risk-neutral  Bayesian  problem \eqref{stat-3} and  the Bayesian-DRO problem \eqref{BDRO}, respectively.
Note that in both cases, the function $H(x,\theta)$ is convex in $x$ if $G(x,\xi)$ is convex in $x$.

Let us discuss convergence of random variables $H_x(\theta_N)=H(x,\theta_N)$,  $\theta_N \sim p(\cdot|\bxi^{(N)})$.

\begin{lemma}\label{lem-weak-convg}
Suppose that  Assumption~\ref{assump} holds and   $\Theta^* = \{\theta^*\}$ is the singleton.  Then for any upper semi-continuous\footnote{Recall that function $h(\theta)$ is said to be
upper semi-continuous if $h(\theta)\ge \limsup_{\theta'\to \theta}h(\theta')$ for $\theta\in \Theta$. Of course, any continuous function is upper semi-continuous.}
function
 $h:\Theta\to\bbr$  it follows that
\begin{equation}\label{limcont}
\lim_{N\to\infty}\int_{\Theta} h(\theta)p(\theta|\bxi^{(N)})d\theta = h(\theta^*), \; {\rm  w.p.1}.
\end{equation}
\end{lemma}

\begin{proof}
Let  $\e>0$  and consider $\gamma_\e:=\sup_{\theta\in U_\e} h(\theta)-h(\theta^*)$.
By the definition \eqref{vsets} we have that $V_\e\cup U_\e=\Theta$.  Note that  since $\theta^*\in U_\e$,  we have that $\gamma_\e\ge 0$. Note also that since function $h(\theta)$ is upper semi-continuous, it attains its maximum over $\theta\in \Theta$, and thus the constant
 $$\lambda:=\sup_{\theta\in \Theta} \{h(\theta)-h(\theta^*)\}$$ is finite (and non-negative).
Then
  we can write
\begin{eqnarray*}
 \left |  \int_{\Theta}h(\theta)p(\theta|\bxi^{(N)})d\theta-h(\theta^*)\right|
 &=&
 \left |  \int_{\Theta}h(\theta)p(\theta|\bxi^{(N)})d\theta-h(\theta^*)
\int_{\Theta} p(\theta|\bxi^{(N)})d\theta
  \right|\\
 &=& \left |  \int_{U_{\e}}\big(h(\theta)-h(\theta^*)\big)p(\theta|\bxi^{(N)})d\theta
+
\int_{V_e} \big(h(\theta)-h(\theta^*) p(\theta|\bxi^{(N)})d\theta\right|\\
 \\
&\le & \gamma_\e  \int_{U_{\epsilon}}p(\theta|\bxi^{(N)})d\theta  + \lambda
 \int_{V_\e} p(\theta|\bxi^{(N)})d\theta
 \\
 &\le & \gamma_\e  + \lambda
 \int_{V_\e} p(\theta|\bxi^{(N)})d\theta.
 \end{eqnarray*}
 By \eqref{convest}  the term $\int_{V_\e} p(\theta|\bxi^{(N)})d\theta$ can be arbitrarily small   w.p.1  for $N$  large enough.   Since
 $h(\cdot)$ is upper semi-continuous
 and   $U_\e$ shrinks to   $\{\theta^*\}$ as $\e\downarrow 0$, we have that
 $
 \limsup_{\e\downarrow 0} \gamma_\e\le 0.
 $
 Because  $\gamma_\e\ge 0$, it follows that
  $\gamma_\e$ tends to zero as $\e\downarrow 0$.
Consequently the assertion \eqref{limcont} follows.
\end{proof}

In both settings  of \eqref{funch}  it can be verified under standard regularity conditions that $H_x(\cdot)$ is upper semi-continuous  on  $\Theta$. Indeed,  in the risk-neutral case we have
\begin{equation}\label{uppersem-1}
   \lim_{\theta'\to\theta} H_x(\theta')=\lim_{\theta'\to\theta}\int  G_x(\xi) f(\xi|\theta')d\xi=
\int \lim_{\theta'\to\theta} G_x(\xi) f(\xi|\theta')d\xi=H_x(\theta),
\end{equation}
i.e., $H_x(\cdot)$ is  continuous,
provided that $f(\xi|\theta)$  is continuous in $\theta\in \Theta$  and the limit and integral can be interchanged (this can be ensured by the respective dominance condition). In the DRO setting of KL-divergence approach, we have that
  \begin{equation}\label{uppersem-2}
H_x(\theta)=\inf_{\lambda > 0}\left\{
 \lambda \e  +\lambda\ln \bbe_{\xi|\theta}  [ e^{
 G_x/\lambda} ] \right\}.
  \end{equation}
The above function is finite valued  by assumption \eqref{assum}.
Since infimum of a family of continuous functions is  upper semi-continuous, it follows that the above $H_x(\cdot)$ is upper semi-continuous provided that $\bbe_{\xi|\theta}  [ e^{
 G_x/\lambda} ] $ is continuous in $\theta$.

For  $x\in \X$  suppose that $H_x(\cdot)$ is upper semi-continuous on $\Theta$.
Then under the assumptions of Lemma \ref{lem-weak-convg} we have by \eqref{limcont} that
\begin{equation}\label{limcon-3}
\lim_{N\to\infty}\bbe_{\theta_N}[H_x] = H(x,\theta^*), \; {\rm  w.p.1}.
\end{equation}
The above can be viewed as a point-wise LLN for random variables $H_x(\theta_N)$.  Under mild additional assumptions this point-wise LLN can be extended (we will discuss this below) to the respective uniform LLN:
 \begin{equation}\label{uniform}
 \lim_{N\to\infty}\sup_{x\in \X}\big|\bbe_{\theta_N}[H_x] - H(x,\theta^*)\big|=0, \; {\rm  w.p.1}.
 \end{equation}

Now consider the limiting optimization problem
\begin{equation}\label{optlimit}
  \min_{x\in \X} H(x,\theta^*).
\end{equation}

Denote by  $\vv_N$ and $\vv^*$ the optimal value of the respective problems \eqref{optimiz-1} and \eqref{optlimit}, and the corresponding sets
\[
\s_N:=\argmin_{x\in \X}\bbe_{\theta_N} [H_x]\;\;{\rm and}\;\;\s^*:=\argmin_{x\in \X} H(x,\theta^*)
\]
of optimal solutions. Suppose that the optimal value $\vv^*$ of problem \eqref{optlimit} is finite. Then the uniform LLN \eqref{uniform} implies that (e.g.,  \cite[Proposition 5.2]{shapiro2021lectures})
\begin{equation}\label{valconv}
\lim_{N\to\infty} \vv_N =\vv^*\;{\rm w.p.1}.
\end{equation}

Under mild additional conditions, it is possible to show that the uniform LLN implies that\footnote{By $\bbd(A,B)$ we denote the deviation of set $A\subset \bbr^n$ from set $B\subset \bbr^n$, that is $\bbd(A,B):=\sup_{x\in A}\dist (x,B),$ with $\dist (x,B)=\sup_{y\in B}\|x-y\|$.}
\begin{equation}\label{optconv}
\lim_{N\to\infty} \bbd(\s_N,\s^*) =0, \;{\rm w.p.1}
\end{equation}
(see, e.g., \cite[Theorems 5.3 and 5.4]{shapiro2021lectures}). This means that  if $x_N$ is an optimal solution of  problem \eqref{optimiz-1}, then the distance from $x_N$ to $\s^*$ tends to zero  w.p.1. In particular, if $\s^*=\{x^*\}$ is the singleton, then $x_N$ converges to $x^*$ w.p.1. \\

Let us discuss now the  uniform LLN \eqref{uniform}. It is relatively easy to derive the uniform LLN in the following convex case.
\begin{assu} \label{assum-2}
Suppose that the set $\X$ is compact and there is a convex  neighborhood\,\footnote{By the ``neighborhood" we mean that the set $\V$ is open and $\X\subset \V$.} $\V$ of $\X$ such that function $H(\cdot,\theta)$ is finite valued convex on $\V$ for every $\theta\in \Theta$.
\end{assu}

Convexity of $H(\cdot,\theta)$ implies convexity  of the expectation function $\int_\Theta H(\cdot,\theta)p(\theta|\bxi^{(N)})d\theta$. It is known by convex analysis that an extended real valued  convex function is continuous on the interior of its domain. Moreover,  if $f_k:\bbr^n\to \overline{\bbr}$ is a sequence of convex functions and $f:\bbr^n\to \overline{\bbr}$ is a convex function such that its domain has a nonempty interior, and $f_k(x)$ converges to $f(x)$ for all $x$ in a dense subset of $\bbr^n$, then $f_k(\cdot)$ converges uniformly  to $f(\cdot)$ on every compact subset of $\bbr^n$ which does not contain a boundary point of the domain of $f$  (e.g., \cite[Theorem 7.17]{rockafellar2009variational}). By using this result it is not difficult to derive the following uniform LLN (e.g., \cite[Theorem 7.50]{shapiro2021lectures}).

\begin{proposition}
\label{pr-conv}
Suppose that Assumption \ref{assum-2} is fulfilled and   the point-wise LLN \eqref{limcon-3} holds for every $x\in \V$. Then the uniform LLN \eqref{uniform} follows.
\end{proposition}

Without the convexity assumption we need to impose additional conditions. The following is similar to a derivation of the uniform LLN in the standard case  (e.g., \cite[Theorem 7.48]{shapiro2021lectures}).

\begin{theorem}
Suppose that  Assumption~\ref{assump} holds, the set   $\Theta^* = \{\theta^*\}$ is the singleton, the set  $\X$ is compact, and the function $H(x,\theta)$ is continuous on $\X\times \Theta$.
Then the uniform LLN \eqref{uniform} follows.
\end{theorem}

\begin{proof}
For a point $\bar{x}\in \X$, a sequence $\nu_k$ of positive numbers converging to zero and $\V_k:=\{x\in \X:\|x-\bar{x}\|\le \nu_k\}$, consider
\[
\Delta_k(\theta):=\sup_{x\in \V_k}|H(x,\theta)-H(\bar{x},\theta)|,\;\theta\in \Theta.
\]
Since $H(x,\theta)$ is continuous on $\X\times \Theta$ and $\X$ is compact, it follows that $\Delta_k(\cdot)$ is continuous on $\Theta$. Then by Lemma \ref{lem-weak-convg} we have that
\begin{equation}\label{limit-1a}
\lim_{N\to\infty}\bbe_{\theta_N} [\Delta_k]= \Delta_k(\theta^*), \; {\rm  w.p.1}.
\end{equation}
By continuity of $H(\cdot,\theta^*)$, we have that $\Delta_k(\theta^*)$ tends to zero as $k\to\infty$. We also have by Lemma \ref{lem-weak-convg}   that
\begin{equation}\label{limit-2a}
\lim_{N\to\infty}\bbe_{\theta_N} [H_{\bar{x}}]= H(\bar{x},\theta^*), \; {\rm  w.p.1}.
\end{equation}
Furthermore for $x\in \V_k$,
\begin{eqnarray*}
\big| \bbe_{\theta_N}[H_x]-\bbe_{\theta_N}[H_{\bar{x}}]\big|
& \leq & \big| \bbe_{\theta_N}[H_x]- H(\bar{x},\theta^*)\big|+
   \big| \bbe_{\theta_N}[H_{\bar{x}}]- H(\bar{x},\theta^*)\big|\\
 & \leq & \bbe_{\theta_N}[\Delta_k]+
   \big| \bbe_{\theta_N}[H_{\bar{x}}]- H(\bar{x},\theta^*)\big|.
\end{eqnarray*}
It follows that for a given $\epsilon>0$ there is a neighborhood $\W$ of $\bar{x}$ such that w.p.1 for $N$ large enough
\begin{equation}\label{limit-3a}
   \sup_{x\in \X\cap \W} \big| \bbe_{\theta_N}[H_x]-\bbe_{\theta_N}[H_{\bar{x}}]\big|\le \epsilon.
\end{equation}
The proof can be completed now exactly in the same way as in the proof of Theorem 7.48 in \cite{shapiro2021lectures} by using compactness of the set $\X$.
\end{proof}

The assumed continuity of $H(x,\theta)$  on $\X\times \Theta$ can be verified under mild regularity conditions. That is, assume that $G(x,\xi)$ is continuous in $x\in \X$,  $f(\xi|\theta)$  is continuous in $\theta\in \Theta$  and $G_x(\xi) f_\theta(\xi)$, $(x,\theta)\in  \X\times \Theta$,  is dominated by an integrable function. Then in the risk neutral case the  continuity of   $H(x,\theta)$ can be verified  similar to \eqref{uppersem-1}. In the DRO setting, with $H(x,\theta)$ given in \eqref{uppersem-2}, the continuity of   $H(x,\theta)$ also follows since the objective function in the right hand side minimization of problem \eqref{uppersem-2} is strictly convex in $\lambda>0$, and thus  the corresponding  minimizer is unique. By convexity of the objective function, this minimizer is a continuous function of $(x,\theta)\in \X\times \Theta$. Therefore, for $(x,\theta)$ in a neighborhood of a considered point the minimization can be restricted to a bounded (compact) subset of $\bbr_+$, and hence the continuity at the considered point follows.

\subsection{Determination of the ambiguity set size}
\label{sec:size}

We consider how to determine the ambiguity set size $\e$ in the Bayesian-DRO problem \eqref{dro-nest}. Recall that  $Q_*$ denotes the true distribution of $\xi$ with $q_*$ denoting its pdf, and  $\mu:=\bbe_{\xi|\theta}[Z]$ , $\sigma^2:=\var_{\xi|\theta}(Z)$ for   $Z:\Xi\to\bbr$.  The true objective function can be written as
\begin{align*}
\bbe_{Q_*}[Z] &= \mu + \bbe_{\xi|\theta}\left[Z(\xi)\frac{q_*(\xi)-f(\xi|\theta)}{f(\xi|\theta)} \right] \\
&= \mu + \bbe_{\xi|\theta}\left[(Z(\xi)-\mu)\frac{q_*(\xi)-f(\xi|\theta)}{f(\xi|\theta)} \right],
\end{align*}
where the second equality uses the fact $\bbe_{\xi|\theta}\left[\frac{q_*(\xi)-f(\xi|\theta)}{f(\xi|\theta)} \right] = 0$.
Applying Cauchy-Schwartz inequality to the right hand side of the equation above, we have
\begin{align*}
\bbe_{Q_*}[Z] \leq \mu + \sigma \bbe_{\xi|\theta}\left[\left(\frac{q_*(\xi)-f(\xi|\theta)}{f(\xi|\theta)}\right)^2 \right]^{1/2},
\end{align*}
where the last term can be simplified as
$$
\bbe_{\xi|\theta} \left[\left(\frac{q_*(\xi)-f(\xi|\theta)}{f(\xi|\theta)}\right)^2 \right] = \bbe_{Q_*}\left[\frac{q_*(\xi)}{f(\xi|\theta)}\right]-1.
$$

If we let $2\e = \bbe_{Q_*}\left[\frac{q_*(\xi)}{f(\xi|\theta)}\right]-1$, then by \eqref{appr-2} we have
\begin{equation}\label{upper-bound}
\bbe_{Q_*}[Z] \leq \mu + \sigma\sqrt{2\e } \approx \min_{\lambda > 0}\left\{
 \lambda \e  +\lambda\ln \bbe_{\xi|\theta} \big[ e^{Z/\lambda}\big]\right\},
\end{equation}
which implies the  objective value of the Bayesian-DRO problem \eqref{dro-nest} is an upper bound on the true objective value. Note here $\e$ depends on $\theta$.

A plausible idea of choosing the ambiguity set size is to make sure the ambiguity set contains the true distribution. That is, we would set
$$
\epsilon(\theta) = D_{KL}(q_*\|f_\theta).
$$

When $q_*$ is close to $f_\theta$, we can write $D_{KL}(q_*\|f_\theta)\approx \bbe_{Q_*}\left[\frac{q_*(\xi)}{f(\xi|\theta)}\right]-1$. However, \eqref{upper-bound} shows even choosing $\epsilon$ half of the size, i.e. $\epsilon = \left(\bbe_{Q_*}\left[\frac{q_*(\xi)}{f(\xi|\theta)}\right]-1\right)/2$, the Bayesian-DRO objective is still an upper bound on the true objective, which indicates this choice of ambiguity set size might be too conservative. Moreover, since $q_*$ is unknown and has to be replaced by a continuous approximation of its empirical distribution, the number of samples required to achieve a certain approximation accuracy grows exponentially in dimension, which makes this method impractical in high dimension.

Now we consider a different method, which is inspired by \cite{Blanchet:2019Wasserstein}. We choose the ambiguity set to be the minimum KL ball containing at least one distribution under which the corresponding problem has the same optimal solution as the true problem. More specifically, we define a set of distributions as
$$
\mathcal{Q}(x^*) := \{Q: x^*\in\argmin_x\bbe_{Q}[G(x,\xi)]\},
$$
where $x^*$ is an optimal solution to the true problem.
When $G(x,\xi)$ is convex in $x$ and $x^*$ is an interior point of $\X$, we can simplify by the first-order optimality condition,
$$
\mathcal{Q}(x^*) = \{Q: \bbe_{Q}[\nabla_x G(x^*,\xi)]=0\}.
$$

In general, we can represent the condition in $\mathcal{Q}(x^*)$ by KKT conditions.
Clearly, $Q_*\in \mathcal{Q}(x^*)$, i.e., the true distribution falls in the set $\mathcal{Q}(x^*)$. Now we set the ambiguity set size by
minimizing the KL divergence from $\mathcal{Q}(x^*)$ to $f_{\theta}$:
\begin{equation} \label{eqn-ambiguityKLball}
\hat{\epsilon}(\theta) = \min_{q\in\mathcal{Q}(x^*)} D_{KL}(q||f_\theta).
\end{equation}

We do not know the optimal solution $x^*$, so in implementation we can replace $x^*$ by the empirical optimal solution $\hat{x}_N$, which is the optimal solution to the SAA problem
$\min_{x\in\mathcal{X}}\bbe_{\hat{Q}_N}[G(x,\xi)]$, where $\hat{Q}_N$ is the empirical distribution of the data $\bxi^{(N)}$. Since $\hat{x}_N - x^* = O_p(N^{-1/2})$ under certain regularity conditions, in particular if the true optimal $x^*$ is unique (see Section 5.1 of \cite{shapiro2021lectures}), and under mild conditions $\epsilon(\theta, x) = \min_{q\in\mathcal{Q}(x)} D_{KL}(q||f_\theta)$ is a smooth function in $x$, one can expect that  $\epsilon(\theta, \hat{x}_N)$ is a good approximation of $\hat{\epsilon}(\theta)$.

\setcounter{equation}{0}
\section{Numerical Experiments}
\label{sec-numeric}
In this section, we demonstrate the performance of Bayesian-DRO on problems of one-dimension and multi-dimension with randomness having continuous and finite support respectively. The Bayesian-DRO problem \eqref{dro-nest} is restated as follows:
\begin{align}
    \min _{x \in \mathcal{X}} \mathbb{E}_{\theta_{N}}\left[\inf _{\lambda>0}\left\{\lambda \epsilon+\lambda \ln \mathbb{E}_{\xi \mid \theta}\left[e^{G_{x} / \lambda}\right]\right\}\right],
\label{eq: numerical_BayesianDRO}
\end{align}
where $N$ is the number of data points, $G_x$ stands for the cost function $G(x,\xi)$. In implementation, we apply SAA (e.g., \cite{shapiro2021lectures}) to solve problem \eqref{eq: numerical_BayesianDRO}. We generate 100 samples of $\theta$ from the posterior distribution $p\left(\theta \mid \boldsymbol{\xi}^{(N)}\right)$ and 100 samples of $\xi$ from the reference distribution $f(\xi | \theta)$ conditioned on each sampled $\theta$. We compare the following approaches.

\begin{enumerate}
    \item[(1)] Bayesian-DRO, with pre-specified ambiguity set size $\epsilon$, which varies in a certain range.

    \item[(2)] Bayesian-DRO, with ambiguity set size $\epsilon_{1}(\theta)=D_{K L}\left(q_{*} \| f(\cdot ; \theta)\right)$, where the unknown true distribution $q_{*}$ is estimated by the empirical distribution of the data.
    \item[(3)] Bayesian-DRO, with ambiguity set size $\epsilon_2(\theta)=\frac{\epsilon_1(\theta)}{2}$.  It halves $\e_1$ to reduce the over-estimation, as shown in Section~\ref{sec:size}.
    \item[(4)] Bayesian-DRO, with ambiguity set size $\epsilon_3(\theta)$, 
    that is, solving problem \eqref{eqn-ambiguityKLball} with $x^{*}$ replaced by the empirical optimal solution to the SAA problem $\min _{x \in \mathcal{X}} \mathbb{E}_{\hat{Q}_{N}}[G(x, \xi)]$, where $\hat{Q}_N$ is the empirical distribution.
    \item[(5)] Bayesian average, that is, solving the Bayesian average problem \eqref{stat-3}, which is the risk-neutral Bayesian average and is equivalent to letting $\epsilon=0$ in Bayesian-DRO.
    \item[(6)] Empirical approach, that is, solving the SAA problem $\min _{x \in \mathcal{X}} \mathbb{E}_{\hat{Q}_{N}}[G(x, \xi)]$.
    \item[(7)] When the distribution of $\xi$ has a finite support $\{\xi_1,\ldots,\xi_m\}$, we compare with Empirical-DRO (KL) in \cite{gotoh2021calibration}. Specifically, we solve the following optimization problem:
    \begin{align*}
        \min_{x \in \mathcal{X}} \max_{Q} \mathbb{E}_{Q} [G(x,\xi)], \quad \text{s.t.} \quad \sum_{i=1}^{m} q_i \log\left(\frac{q_i}{\hat{p}_i}\right) \leq \epsilon, \sum_{i: \hat{p}_i > 0} q_i = 1, q_i \geq 0,
    \end{align*}
    where $Q=[q_1,\cdots,q_m]$, $\hat{p}_i$ is the probability mass on $\xi_i$ in the empirical distribution.
    \item[(8)] 
    We also compare with the DRO - Wasserstein.
That is, we solve the following optimization problem:
    \begin{align}
        \min_{x \in \mathcal{X}} \max_{Q} \mathbb{E}_{Q}[G(x,\xi)], \quad \text{s.t.} \quad W_p(Q,\hat{Q}_N) \leq \tilde{\epsilon},
    \label{eq: numerical_Wasserstein}
    \end{align}
    where $W_p(Q,\hat{Q}_N)$ is the Wasserstein distance of order $p$ between $Q$ and the empirical distribution $\hat{Q}_N$, and $\tilde{\epsilon}$ is the ambiguity set size. The dual of \eqref{eq: numerical_Wasserstein} is given by \cite{Esfahani-Kuhn,Blanchet:2019Wasserstein,gao2016distributionally}:
    \begin{align*}
        \min_{x \in \mathcal{X}, \lambda \geq 0} \lambda \tilde{\epsilon}^{p} + \frac{1}{N} \sum_{i=1}^{N} \sup_{\xi \in \Xi} [G(x,\xi)-\lambda d(\xi, \hat{\xi}_i)^{p}],
    \end{align*}
    where $\Xi$ is the space of $\xi$, $d(\xi, \hat{\xi}_i)$ is the metric (or distance function) between two points $\xi$ and $\hat{\xi}_i$, and $\{\hat{\xi}_i\}_{i=1}^{N}$ are the data points. In our experiments, we consider Wasserstein distance  of order $p = 1, 2$, and the metric is chosen to be Euclidean norm. It is shown in \cite{Esfahani-Kuhn} that, under mild assumptions, the distributionally robust optimization problems over Wasserstein balls can be reformulated as finite convex programs.
\end{enumerate}

When the randomness has finite support, we choose the prior distribution in Bayesian-DRO and Bayesian average to be an uninformative Dirichlet distribution on $\theta$. Sampling from a Dirichlet posterior distribution given the data is the same as Bayesian bootstrapping \cite{LamZhou2008}. Please note that in this case, we implicitly choose the parameterized family to contain all discrete distributions on the support, which is the correct model. Numerical results for finite-support examples are shown in the Online Appendix.

When the distribution of $\xi$ is continuous, we compute the ambiguity set sizes in Bayesian-DRO 
with the following implementation details.
\begin{itemize}
    \item In Bayesian-DRO with ambiguity set size $\epsilon_{1}(\theta)$ and $\epsilon_{2}(\theta)$, the KL divergence from the empirical distribution to the reference distribution is estimated using the estimation method in \cite{perez2008kullback}. Specifically, we compute the empirical cumulative distribution function (cdf) given the data, construct linear interpolation of the empirical cdf, and then we use the finite difference method to compute the estimated KL divergence as:
    \begin{align*}
        \widehat{D}_{\mathrm{KL}}(Q \| f(\cdot ; \theta))=\frac{1}{N} \sum_{i=1}^{N} \log \left(\frac{\delta P_{c}\left(\hat{\xi}_{i}\right)}{\Delta f\left(\hat{\xi}_{i} ; \theta\right)}\right),
    \end{align*}
    where $\left\{\hat{\xi}_{i}\right\}_{i=1}^{N}$ are the data points, $P_{c}$ is the linear interpolation of the empirical cdf, $\delta P_{c}\left(\hat{\xi}_{i}\right)=P_{c}\left(\hat{\xi}_{i}\right)-P_{c}\left(\hat{\xi}_{i}-\Delta\right)$, $\Delta<\min _{i}\left\{\hat{\xi}_{i}-\hat{\xi}_{i-1}\right\}$.
    \item In Bayesian-DRO with ambiguity set size $\epsilon_{3}(\theta)$, to compute the minimum KL ball, we conduct Monte Carlo sampling from $f(\xi|\theta)$. Essentially, we employ SAA to solve the problem
    \begin{align*}
        \min_{q} & \frac{1}{L} \sum_{i=1}^{L} \log(\frac{q(\xi_i)}{f(\xi_i|\theta)}) \frac{q(\xi_i)}{f(\xi_i|\theta)} \\
        & \text{s.t. } \frac{1}{L} \sum_{i=1}^{L} \frac{q(\xi_i)}{f(\xi_i|\theta)} = 1, \quad \frac{1}{L} \sum_{i=1}^{L} \nabla_x G(x^{*},\xi_i) \frac{q(\xi_i)}{f(\xi_i|\theta)} = 0,  \quad q(\xi_i)\geq 0,
    \end{align*}
    where $\xi_1,\cdots,\xi_L$ are $L=100$ samples drawn from $f(\xi|\theta)$. We solve this optimization problem using Gurobi 9.1 with Python 3.7 API and scipy package in Python. Algorithmic description of this approach can be found in Algorithm~\ref{algorithm} in the appendix.
\end{itemize}

We evaluate the performance of each algorithm following the procedure in \cite{gotoh2021calibration}, as follows. All algorithms are run for $K=200$ replications. In each replication $j=1,\cdots, K$, we collect $N$ data points $\hat{\xi}_1,\cdots,\hat{\xi}_N$ drawn i.i.d. from the true distribution $\mathbb{P}_{\theta^{c}}$. Then we run each algorithm with the same data set and obtain its optimal solution, denoted by $x^{(j)}(\epsilon)$, where $\epsilon$ is the corresponding ambiguity set size.
We then compute $\mu^{(j)}(\epsilon)=\mathbb{E}_{\mathbb{P}_{\theta^{c}}}[G(x^{(j)}(\epsilon), \xi)]$ and $v^{(j)}(\epsilon)=\operatorname{Var}_{\mathbb{P}_{\theta^{c}}}[G(x^{(j)}(\epsilon), \xi)]$, i.e.,  the (mean and variance) performance of the obtained solutions under the true system. The out-of-sample mean and variance are then approximated using these $K=200$ replications, with $\hat{\mu}_{N}(\epsilon)=\frac{1}{K}\sum_{j=1}^{K}\mu^{(j)}(\epsilon)$ and $\hat{v}_{N}(\epsilon)=\frac{1}{K}\sum_{j=1}^{K}v^{(j)}(\epsilon)+\frac{1}{K-1}\sum_{j=1}^{K}(\mu^{(j)}(\epsilon)-\hat{\mu}_{N}(\epsilon))^2$.

\subsection{One-dimensional Newsvendor with Continuous Randomness}
In this subsection, we run experiments on a one-dimensional newsvendor problem when the randomness $\xi$ has a continuous distribution and the data all come from the true distribution (see \cite{porteus1990stochastic} for a review on newsvendor models). We summarize notations used in the classical newsvendor problem as follows.
\begin{itemize}
    \item $x$: order amount, assumed to be in $[0, M]$, $M$ is the maximal order amount.
    \item $\xi$: random customer demand.
    \item $b$: backorder cost per unit.
    \item $h$: holding cost per unit.
    \item $c$: ordering cost per unit.
\end{itemize}

The cost function is given by $G(x,\xi)=h (x-\xi)^{+} + b(\xi-x)^{+}+cx$, where $(\cdot)^{+}=\max (\cdot, 0)$. We assume the customer demand $\xi \in \Xi$, where $\Xi=(0,\infty)$.
Parameters used in the newsvendor problem are summarized as follows: maximal ordering amount $M=50$, backorder cost $b=8$, holding cost $h=3$, ordering cost $c=0$.

\begin{figure}[!tbh]
    \centering
    \begin{subfigure}[t]{0.45\textwidth}
        \centering     \includegraphics[width=\textwidth]{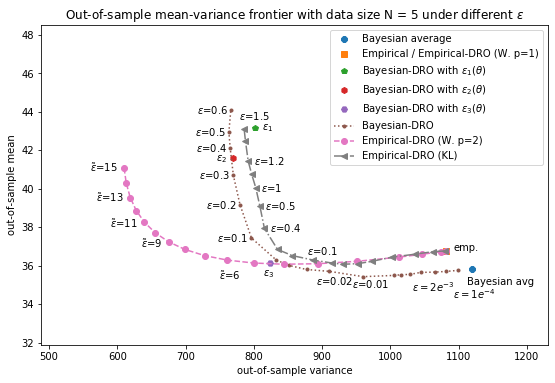}
        \caption{$N=5$.}
        \label{fig: newsvendor_continuous_mis_non_m5}
    \end{subfigure}\hfill%
    \begin{subfigure}[t]{0.45\textwidth}
        \centering     \includegraphics[width=\textwidth]{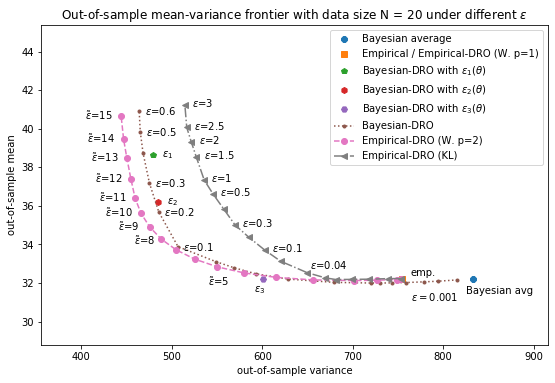}
        \caption{$N=20$.}
        \label{fig: newsvendor_continuous_mis_non_m20}
    \end{subfigure}
    \caption{Newsvendor with continuous support: out-of-sample mean-variance frontiers of different algorithms under different $\epsilon$ values. Data size $N$ is $5$ and $20$ respectively. Bayesian-DRO has model mis-specification.}
    \label{fig: newsvendor_continuous_mis_non}
\end{figure}

\begin{table}[h!]
\centering
{\small
\begin{tabular}{|c|c|c|c|c|c|c|}
\hline
N=5 & $\epsilon_1$ & $\epsilon_2$ & $\epsilon_3$ & Bayesian avg & empirical & true \\
\hline
$\epsilon$ value & 0.58(0.04) & 0.29(0.02) & 0.07(0.01) & - & - & - \\
\hline
solution & 26.13(0.80) & 24.04(0.65) & 18.15(0.38) & 15.46(0.31) & 16.44(0.38) & 17.41 \\
\hline
mean & 43.14(0.33) & 41.62(0.71) & 36.16(0.53) & 35.81(0.31) & 36.77(0.25) & 30.96 \\
\hline
variance & 802.24(2.39) & 769.72(2.11) & 823.97(2.56) & 1119.30(2.73) & 1082.21(2.93) & 640.59 \\
\hline
\end{tabular}
}
{\small
\begin{tabular}{|c|c|c|c|c|c|c|}
\hline
N=20 & $\epsilon_1$ & $\epsilon_2$ & $\epsilon_3$ & Bayesian avg & empirical & true \\
\hline
$\epsilon$ value & 0.34(0.02) & 0.17(0.01) & 0.03(0.00) & - & - & - \\
\hline
solution & 24.36(0.38) & 22.13(0.30) & 18.30(0.17) & 16.16(0.15) & 16.97(0.17) & 17.41 \\
\hline
mean & 38.62(0.54) & 36.19(0.60) & 33.22(0.07) & 32.22(0.06) & 32.20(0.07) & 30.96 \\
\hline
variance & 478.74(1.62) & 485.15(1.49) & 601.01(1.95) & \hspace{0.2cm}832.66(2.38) & \hspace{0.2cm}754.11(2.22) & 640.59\\
\hline
\end{tabular}
}
\\
\caption{Newsvendor with continuous support: out-of-sample performance of variants of Bayesian-DRO with model mis-specification. Data size $N$ is $5$ and $20$ respectively.}
\label{Table: newsvendor_continuous_mis_non}
\end{table}

In the first experiment, we test the performance of our proposed algorithms under model mis-specification. Specifically, the true distribution of the customer demand is normal distribution with mean 10 and variance 100 truncated above 0. In Bayesian-DRO, we choose the parametric family $f(\xi|\theta)$ to be the exponential distribution with rate parameter $\theta$. To have closed-form posterior update, we use the conjugate prior of gamma distribution with parameter $(1,1)$.  Please note this choice of prior distribution is only for computational convenience. If the Bayesian updating does not admit closed-form posterior, we may use Monte Carlo simulation, such as Markov Chain Monte Carlo (MCMC) methods, to draw samples from the posterior;  we only need sample average approximation of the expectations when solving the Bayesian-DRO problem. Figure~\ref{fig: newsvendor_continuous_mis_non} shows the out-of-sample mean-variance frontiers (with varying $\epsilon$ values) of different algorithms for data size $N=5$ and 20. For the empirical approach (abbreviated as empirical), Bayesian average approach (abbreviated as Bayesian average), and Bayesian-DRO with calibrated ambiguity set size $\epsilon_1(\theta), \epsilon_2(\theta),\epsilon_3(\theta)$, their performance is denoted by one point (not a frontier) in the figure. Note that for Empirical-DRO with Wasserstein distance (abbreviated as W. in the figure) of order $p=1$, it is equivalent to the empirical approach (see Remark 6.7 in \cite{Esfahani-Kuhn} and Theorem 3.2 in \cite{lee2021data}) and is independent of the ambiguity set size. Table~\ref{Table: newsvendor_continuous_mis_non} shows the out-of-sample performance of variants of Bayesian-DRO when data size is $N=5$ and 20 respectively; solving the true problem (abbreviated as true) is included as a benchmark for all compared algorithms; standard errors of the average $\epsilon$ values, obtained solutions, and the out-of-sample performances are shown within the parentheses in the table.

\begin{figure}[!tbh]
    \centering
    \begin{subfigure}[t]{0.45\textwidth}
        \centering     \includegraphics[width=\textwidth]{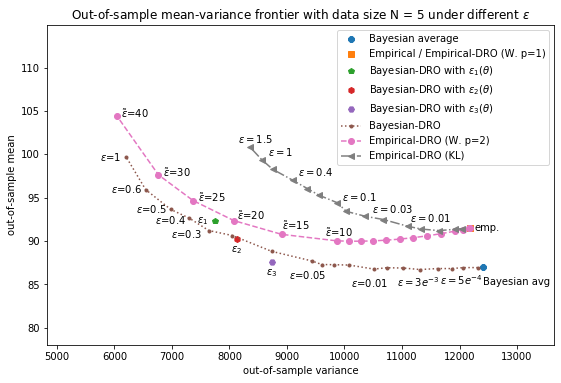}
        \caption{$N=5$.}
        \label{fig: newsvendor_continuous_expon_non_m5}
    \end{subfigure}\hfill%
    \begin{subfigure}[t]{0.45\textwidth}
        \centering     \includegraphics[width=\textwidth]{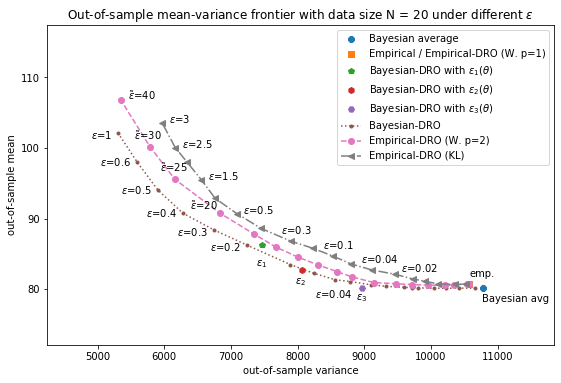}
        \caption{$N=20$.}
        \label{fig: newsvendor_continuous_expon_non_m20}
    \end{subfigure}
    \caption{Newsvendor with continuous randomness: out-of-sample mean-variance frontiers of different algorithms under different $\epsilon$ values. Data size $N$ is $5$ and $20$ respectively. Bayesian-DRO chooses the correct model.}
    \label{fig: newsvendor_continuous_expon_non}
\end{figure}

\begin{table}[h!]
\centering
{\footnotesize
\begin{tabular}{|c|c|c|c|c|c|c|}
\hline
N=5 & $\epsilon_1$ & $\epsilon_2$ & $\epsilon_3$ & Bayesian avg & empirical & true \\
\hline
$\epsilon$ value & 0.35(0.04) & 0.17(0.02) & 0.09(0.00) & - & - & - \\
\hline
solution & 31.91(1.03) & 28.73(0.96) & 27.64(0.83) & 23.30(0.73) & 25.63(1.01) & 25.99 \\
\hline
mean & 92.29(0.90) & 90.25(0.85) & 87.56(0.68) & 86.99(0.72) & 91.50(1.16) & 77.66 \\
\hline
variance & 7743.27(44.87) & 8134.65(44.00) & 8743.18(37.55) & 12409.66(37.35) & 12184.37(42.12) & 9760.90 \\
\hline
\end{tabular}
}
{\footnotesize
\begin{tabular}{|c|c|c|c|c|c|c|}
\hline
N=20 & $\epsilon_1$ & $\epsilon_2$ & $\epsilon_3$ & Bayesian avg & empirical & true \\
\hline
$\epsilon$ value & 0.18(0.01) & 0.09(0.01) & 0.03(0.00) & - & - & - \\
\hline
solution & 33.11(0.87) & 30.95(0.84) & 28.32(0.46) & 24.71(0.39) & 25.31(0.44) & 25.99 \\
\hline
mean & 86.21(0.88) & 82.65(0.83) & 80.18(0.29) & 80.15(0.22) & 80.70(0.25) & 77.66 \\
\hline
variance & 7470.54(34.99) & 8067.69(36.30) & 8968.67(19.86) & 10773.79(19.98) & 10570.40(22.20) & 9760.90 \\
\hline
\end{tabular}
\\
}
\caption{Newsvendor with continuous randomness: out-of-sample performance of variants of Bayesian-DRO without model mis-specification. Data size $N$ is $5$ and $20$ respectively.}
\label{Table: newsvendor_continuous_expon_non}
\end{table}

In the second experiment, we test the performance of our proposed algorithms without model mis-specification. Specifically, the true distribution of the customer demand is exponential distribution with mean 20. We choose the parametric family $f(\xi|\theta)$ to be the correct model, i.e., the exponential distribution with rate parameter $\theta$. Figure~\ref{fig: newsvendor_continuous_expon_non} shows the out-of-sample mean-variance frontiers (with varying $\epsilon$ values) of different algorithms for data size $N=5$ and 20. Table~\ref{Table: newsvendor_continuous_expon_non} shows the out-of-sample performance of variants of Bayesian-DRO when data size is $N=5$ and 20 respectively.

We have the following observations from the two experiments above.

\begin{enumerate}
    \item[(1)] \textbf{Trade-off between out-of-sample mean and variance}: both Bayesian-DRO and Empirical-DRO show the trade-off. As the ambiguity set size $\epsilon$ grows larger, the out-of-sample mean deteriorates, which trades for more robustness in terms of smaller out-of-sample variance. Empirical approach is equivalent to Empirical-DRO with $\epsilon=0$, and Bayesian average is equivalent to Bayesian-DRO with $\epsilon=0$. Therefore, empirical approach and Bayesian average produce solutions with larger out-of-sample variance and smaller out-of-sample mean compared to Empirical-DRO and Bayesian-DRO respectively.
    \item[(2)] \textbf{Model mis-specification affects the performance of Bayesian-DRO}: when there is model mis-specification, Bayesian-DRO underperforms Empirical-DRO with Wasserstein distance of order $p=2$, as can be seen from the worse mean-variance frontier in Figure~\ref{fig: newsvendor_continuous_mis_non}. If we choose the correct model, Bayesian-DRO outperforms Empirical-DRO with Wasserstein distance of order $p=2$, as can be seen from Figure~\ref{fig: newsvendor_continuous_expon_non}. This is expected, since a poorly chosen model, which serves as the reference distribution of the ambiguity set in Bayesian-DRO, deteriorates the performance of Bayesian-DRO. However, the ambiguity set in Bayesian-DRO still provides robustness against model mis-specification, as it can be seen from Figure~\ref{fig: newsvendor_continuous_mis_non}  that Bayesian-DRO (with $\epsilon_3$) has about the same out-of-sample mean but much smaller variance than Bayesian average (which is equivalent to $\epsilon=0$ in Baysian-DRO).
    \item[(3)] \textbf{Bayesian-DRO outperforms Empirical-DRO with KL divergence}: in almost all the experiments, the mean-variance frontier of Bayesian-DRO dominates that of Empirical-DRO (KL). The reason is because the ambiguity sets of Bayesian-DRO contain distributions supported on the domain of the randomness if the prior distribution is chosen to cover the domain, whereas the Empirical-DRO with KL divergence  only allows probability distributions in the ambiguity set that are absolutely continuous with respect to the empirical distribution (i.e., the observed data points) and leaves out distributions supported on the unobserved domain. 
    \item[(4)] \textbf{Parameter-dependent ambiguity set size outperforms pre-specified ones}: for Bayesian-DRO with parameter-dependent ambiguity set size $\epsilon_2(\theta), \epsilon_3(\theta)$, the out-of-sample performances are better compared to Bayesian-DRO with pre-specified ambiguity set size (i.e., fixed $\epsilon$ for all $\theta$). It shows we can gain  better performance  for Bayesian-DRO by tuning an appropriate parameter-dependent ambiguity set size, although this incurs more computational cost.
    \item[(5)] \textbf{Large data size reduces model uncertainty}: as expected, solutions of all the methods become more stabilized (smaller variance) as data size increases. In particular, solution of the empirical approach gets closer to the true optimal solution with more data.
\end{enumerate}

\subsection{Multi-dimensional Newsvendor with Continuous Randomness}
In this subsection, we consider a three-dimensional newsvendor problem with multi-items, where the newsvendor sells three kinds of items (see \cite{turken2012multi} for a review on newsvendor models). Assume the customer demands for each kind of item are independent and
follow normal distribution with mean 10, 12, 15 and standard deviation 20, 20, 20 respectively, truncated above 0. The objective function is given by: $
    G(x, \xi)=\sum_{i=1}^{3} h_{i}\left(x_{i}-\xi_{i}\right)^{+}+b_{i}\left(\xi_{i}-x_{i}\right)^{+}.
$ We set $h_{i}=3, b_{i}=8 \text { for } i=1,2,3$.

\begin{table}[h!]
\centering
{\footnotesize
\begin{tabular}{|c|c|c|c|c|c|c|}
\hline
N=10 & $\epsilon_1$ & $\epsilon_2$ & $\epsilon_3$ & Bayesian avg & empirical & true \\
\hline
$\epsilon$ value & 1.05(0.03) & 0.53(0.02) & 0.17(0.01) & - & - & - \\
\hline
sol error & 20.95(0.50) & 18.84(0.43) & 11.67(0.32) & 10.33(0.28) & 12.37(0.39) & 0.00 \\
\hline
mean & 254.72(2.09) & 238.56(1.83) & 198.21(0.78) & 184.13(0.76) & 190.28(0.95) & 171.28 \\
\hline
variance & 4585.39(35.09) & 4759.10(28.18) & 5845.42(16.55) & 10030.76(20.25) & 8613.13(20.35) & 7066.05 \\
\hline
\end{tabular}
}
\caption{Multi-dimensional newsvendor with continuous randomness: out-of-sample performance of variants of Bayesian-DRO that has model mis-specification. Data size $N$ is 10.}
\label{Table: newsvendor_continuous_high}
\end{table}

The parametric distribution we choose is the exponential distribution with rate parameter $\theta_i$ for each customer demand for item $i$. Figure~\ref{fig: newsvendor_continuous_high} shows the out-of-sample mean-variance frontiers (with varying $\epsilon$ values) of different algorithms when data size $N=10$. Table~\ref{Table: newsvendor_continuous_high} shows the out-of-sample performance of variants of Bayesian-DRO when data size $N=10$; in addition to out-of-sample performance, we also show the solution error, which is obtained by calculating each solution's Euclidean distance from the true optimal solution; standard errors of the average $\epsilon$ values, obtained solution error, and the out-of-sample performances are shown within the parentheses in the table. Similar to the one-dimensional newsvendor problem, \textbf{Bayesian-DRO outperforms Empirical-DRO (KL)} in the multi-dimensional newsvendor problem.

\begin{figure}[t]
    \centering
    \includegraphics[width=0.6\textwidth]{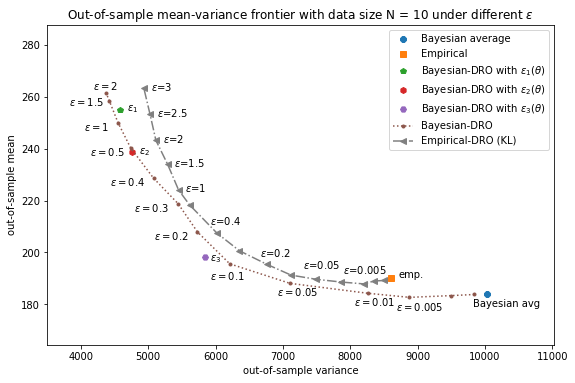}
    \caption{Multi-dimensional newsvendor with continuous randomness: out-of-sample mean-variance frontiers of different algorithms under different $\epsilon$ values. Data size $N$ is 10. Bayesian-DRO has model mis-specification.}
    \label{fig: newsvendor_continuous_high}
\end{figure}

\section{Conclusions and Future Work}
\label{sec-conclusion}
We propose a new formulation, Bayesian Distributionally Robust Optimization (Bayesian-DRO), to address the ambiguity about the probability distribution in static stochastic optimization. Bayesian-DRO takes advantage of Bayesian estimation of parametric distributions and at the same time imposes robustness against the uncertainty introduced by the assumed parametric model. When the ambiguity set is constructed using Kullback-Leibler divergence and the size of the set is small, the robustness of Bayesian-DRO can be interpreted as a trade-off between the posterior mean and standard deviation of the cost function. We show the strong consistency of Bayesian posterior distributions, and subsequently show the convergence of objectives and optimal solutions of Bayesian-DRO problems.  Moreover, we consider several methods of determining the ambiguity set size in Bayesian-DRO.  Our numerical results demonstrate that when data are limited, Bayesian-DRO has superior out-of-sample performance compared to KL-based empirical DRO, the Bayesian-average approach, and the empirical approach; Bayesian-DRO outperforms Wasserstein-based empirical DRO when the parametric family is correctly chosen (i.e., no model mis-specification) but underperforms when there is model mis-specification. More future research is needed to fully understand the connections between these frameworks (Bayesian-DRO, empirical-DRO, BRO) and how to choose a framework for specific data-driven stochastic optimization problems.     


The nature of sequential Bayesian updating makes Bayesian approaches especially amenable to multi-stage (dynamic) settings where data come sequentially in time. One of the future works is to extend Bayesian-DRO to multi-stage stochastic optimization, including multi-stage stochastic programming, stochastic control, and Markov decision processes.

\section*{Acknowledgment}
All authors are grateful for the support by Air Force Office of Scientific Research (AFOSR) under Grant FA9550-22-1-0244. The second and third authors are also grateful for the support by  AFOSR under Grant FA9550-19-1-0283 and National Science Foundation (NSF) under Grant DMS2053489.

\bibliographystyle{plain}
\bibliography{references}

\newpage
\appendix

\section{Supplementary Numerical Experiments}
\subsection{Algorithm 1: Bayesian-DRO with ambiguity set size $\epsilon_3$}
\begin{algorithm}[H]
\SetAlgoLined
\SetKwInOut{Input}{input}\SetKwInOut{Output}{output}
\Input{data points of size $N$, number of $\theta$ samples $N_{\theta}$, number of $\xi$ samples $N_{\xi}$, number of Monte Carlo samples to compute the ambiguity set size $L$}
\Output{optimal solution $x(\epsilon_3)$}
Solve for the SAA solution $x_N^{*}$\;
\For{$i=1 \leftarrow 1$ \KwTo $N_{\theta}$}{
Simulate $\theta_i$ from posterior distribution $p(\theta|\boldsymbol{\xi}^{(N)})$\;
Simulate $\{\xi_j\}_{j=1}^{L}$ from reference distribution $f(\xi|\theta_i)$, solve the optimization problem
\begin{align*}
   & \epsilon_3(\theta_i)=\min_{q}\frac{1}{L} \sum_{j=1}^{L} \log(\frac{q(\xi_j)}{f(\xi_i|\theta_i)}) \frac{q(\xi_i)}{f(\xi_j|\theta)} \\
    & \text{s.t. } \frac{1}{L} \sum_{j=1}^{L} \frac{q(\xi_j)}{f(\xi_j|\theta_i)} = 1, \quad \frac{1}{L} \sum_{j=1}^{L} \nabla_x G(x_N^{*},\xi_j) \frac{q(\xi_j)}{f(\xi_j|\theta_i)} = 0,  \quad q(\xi_j)\geq 0;
\end{align*}
Simulate $\{\hat{\xi}_j\}_{j=1}^{N_{\xi}}$ from reference distribution $f(\xi|\theta_i)$ and store them as dataset $\mathcal{D}_i$\;
}
Solve the Bayesian-DRO problem and obtain the optimal solution $x(\epsilon_3)$
\begin{align*}
    \min _{x \in \mathcal{X}, \lambda_i>0}\left\{\frac{1}{N_{\theta}} \sum_{i=1}^{N_{\theta}}\left( \lambda_{i} \epsilon_3\left(\theta_{i}\right)+\lambda_{i} \log \left(\frac{1}{N_{\xi}} \sum_{\hat{\xi} \in \mathcal{D}_{i}} \exp \left(G(x, \hat{\xi}) / \lambda_{i}\right)\right)\right)\right\}.
\end{align*}
\caption{Bayesian-DRO with ambiguity set size $\epsilon_3$.}
\label{algorithm}
\end{algorithm}

\subsection{One-dimensional Newsvendor with Finite-support Randomness}
In this subsection, we first run experiments on a one-dimensional newsvendor problem when the randomness $\xi$ has a finite support and the data all come from the true distribution. Different from the continuous-support case, the random customer demand is assumed to take discrete values in $\{1,2,\cdots,14,15\}$. The true probability mass $\theta^c \in \Delta_{15}$ is unknown to the decision maker, where $\Delta_{15}$ stands for a probability simplex. Parameters used in the newsvendor problem are summarized as follows. Maximal ordering amount $M=20$, backorder cost $b=10$, holding cost $h=2$, ordering cost $c=3$.

\begin{figure}[!tbh]
    \centering
    \begin{subfigure}[t]{0.45\textwidth}
        \centering     \includegraphics[width=\textwidth]{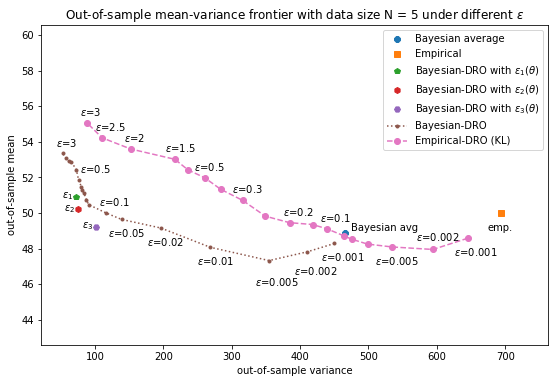}
        \caption{$N=5$.}
        \label{fig: newsvendor_m5_non}
    \end{subfigure}\hfill%
    \begin{subfigure}[t]{0.45\textwidth}
        \centering     \includegraphics[width=\textwidth]{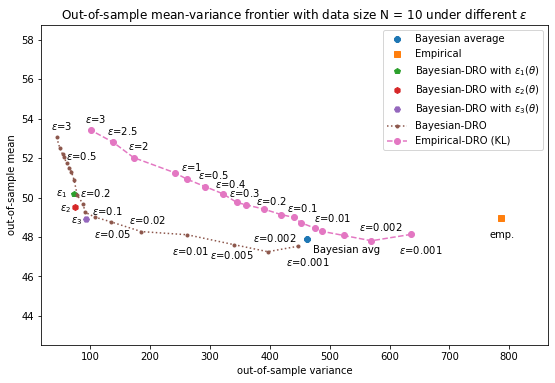}
        \caption{$N=10$.}
        \label{fig: newsvendor_m10_non}
    \end{subfigure}\hfill%
    \begin{subfigure}[t]{0.45\textwidth}
        \centering     \includegraphics[width=\textwidth]{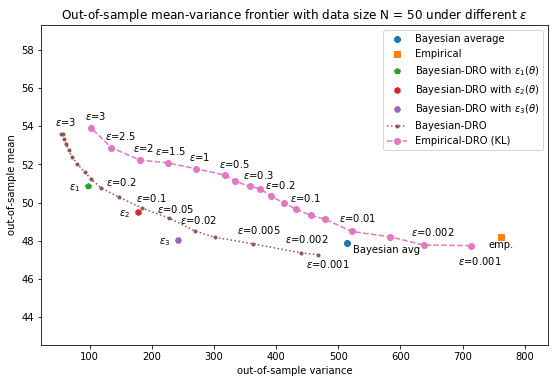}
        \caption{$N=50$.}
        \label{fig: newsvendor_m50_non}
    \end{subfigure}\hfill%
    \begin{subfigure}[t]{0.45\textwidth}
        \centering     \includegraphics[width=\textwidth]{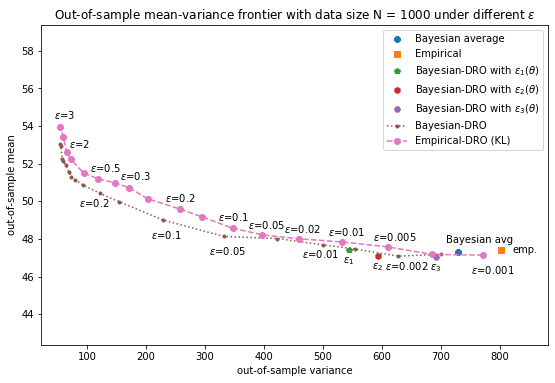}
        \caption{$N=1000$.}
        \label{fig: newsvendor_m1000_non}
    \end{subfigure}
    \caption{Newsvendor with finite-support randomness: out-of-sample mean-variance frontiers of different algorithms under different $\epsilon$ values. Data size varies from $5,10,50$ to $1000$.}
    \label{fig: newsvendor_non}
\end{figure}

Figure~\ref{fig: newsvendor_non} shows the out-of-sample mean-variance frontiers (with varying $\epsilon$ values) of different algorithms for data sizes $N=5$, 10, 50 and 1000. Table~\ref{Table: newsvendor_non} shows the out-of-sample performance of each algorithm when data size is $N=5,10,50,1000$ respectively. Similar to the continuous-support case, \textbf{Bayesian-DRO performs better than Empirical-DRO in most cases}, as the mean-variance frontier of Bayesian-DRO dominates that of empirical-DRO (KL). Note that for a small data size, Empirical-DRO (KL) will only put non-negative probability mass on the support point $\hat{\xi}$ that has been observed in the data. On the other hand, by imposing an appropriate prior (in this problem we use a non-informative Dirichlet prior whose domain is a uniform distribution on the support of $\xi$), Bayesian-DRO can put non-negative probability mass on all the support points. Also note that the mean-variance frontiers of Bayesian-DRO and Empirical-DRO get closer as the data size $N$ goes to infinity due to the reduced model uncertainty.

\begin{table}[!tbh]
\centering
{\small
\begin{tabular}{|c|c|c|c|c|c|c|}
\hline
N=5 & $\epsilon_1$ & $\epsilon_2$ & $\epsilon_3$ & Bayesian avg & empirical & true \\
\hline
$\epsilon$ value & \hspace{0.06cm} 0.99(0.02) & \hspace{0.06cm} 0.50(0.01) & \hspace{0.06cm} 0.14(0.01) & - & - & - \\
\hline
solution & 12.77(0.17) & 12.70(0.17) & 11.29(0.14) & 8.59(0.06) & 8.55(0.24) & 7.00 \\
\hline
mean & 50.92(0.43) & 50.22(0.50) & 49.24(0.47) & 48.90(0.02) & 50.01(0.28) & 47.21 \\
\hline
variance & \hspace{0.06cm} 72.27(2.15) & \hspace{0.06cm} 75.12(2.28) & 101.24(2.38) & 465.08(2.91) & 693.63(3.63) & 770.56 \\
\hline
\end{tabular}
}
{\small
\begin{tabular}{|c|c|c|c|c|c|c|}
\hline
N=10 & $\epsilon_1$ & $\epsilon_2$ & $\epsilon_3$ & Bayesian avg & empirical & true \\
\hline
$\epsilon$ value & \hspace{0.06cm} 0.60(0.01) & \hspace{0.06cm} 0.30(0.01) & \hspace{0.06cm} 0.07(0.01) & - & - & - \\
\hline
solution & 12.28(0.16) & 12.20(0.14) & 10.61(0.10) & 8.57(0.09) & 7.50(0.20) & 7.00 \\
\hline
mean & 50.20(0.37) & 49.54(0.28) & 48.93(0.30) & 47.87(0.05) & 48.94(0.15) & 47.21 \\
\hline
variance & \hspace{0.06cm} 71.99(2.07) & \hspace{0.06cm} 74.97(2.09) & \hspace{0.05cm} 91.97(2.14) & 461.88(2.55) & 786.58(3.56) & 770.56 \\
\hline
\end{tabular}
}
{\small
\begin{tabular}{|c|c|c|c|c|c|c|}
\hline
N=50 & $\epsilon_1$ & $\epsilon_2$ & $\epsilon_3$ & Bayesian avg & empirical & true \\
\hline
$\epsilon$ value & \hspace{0.06cm} 0.14({0.00}) & \hspace{0.06cm} 0.07({0.00}) & \hspace{0.06cm} 0.02({0.00}) & - & - & - \\
\hline
solution & 11.24({0.04}) & 10.91({0.05}) & 9.26({0.06}) & 7.99({0.10}) & 7.36({0.13}) & 7.00 \\
\hline
mean & 50.44({0.05}) & 49.51({0.06}) & 48.22({0.04}) & 47.87({0.04}) & 48.21({0.05}) & 47.21 \\
\hline
variance & 118.65({1.99}) & 178.06({1.88}) & 241.68({1.89}) & 513.35({2.08}) & 761.45({2.86}) & 770.56 \\
\hline
\end{tabular}
}
{\small
\begin{tabular}{|c|c|c|c|c|c|c|}
\hline
N=1000 & $\epsilon_1$ & $\epsilon_2$ & $\epsilon_3$ & Bayesian avg & empirical & true \\
\hline
$\epsilon$ value & 0.006({0.00}) & 0.003({0.00}) & 0.001({0.00}) & - & - & - \\
\hline
solution & 8.08({0.01}) & 7.82({0.01}) & 7.25({0.03}) & 7.18({0.05}) & 6.93({0.06}) & 7.00 \\
\hline
mean & 47.42({0.03}) & 47.12({0.03}) & 47.03({0.02}) & 47.32({0.01}) & 47.39({0.02}) & 47.21\\
\hline
variance & 544.42({0.92}) & 593.85({1.29}) & 691.00({1.15}) & 728.15({1.47}) & 801.66({1.50}) & 770.56 \\
\hline
\end{tabular}
}
\\
\caption{Newsvendor with finite-support randomness: out-of-sample performance of variants of Bayesian-DRO. Data size $N$ varies from 5, 10, 50 to 1000.}
\label{Table: newsvendor_non}
\end{table}

Next, we consider a contaminated data model, where 80$\%$ data are generated from the true distribution and 20$\%$ data are generated from an arbitrary distribution. In particular, the arbitrary distribution is randomly generated (specified by its probability mass) and is different in each replication. Figure~\ref{fig: newsvendor_con} shows the out-of-sample mean-variance frontiers  (with varying $\epsilon$ values) of different algorithms for data size $N=5$ and $50$. Table~\ref{Table: newsvendor_con} shows the out-of-sample performance of all variants of Bayesian-DRO when data size is $N=5$ and $50$ respectively. Similar to the non-contaminated case, \textbf{Bayesian-DRO outperforms other benchmarks} even when data are contaminated. Note that the solution of the empirical approach does not get closer to the true optimal solution with more data, since part of the data are not from the true distribution and possibly become outliers.

\begin{figure}[!tbh]
    \centering
    \begin{subfigure}[t]{0.45\textwidth}
        \centering     \includegraphics[width=\textwidth]{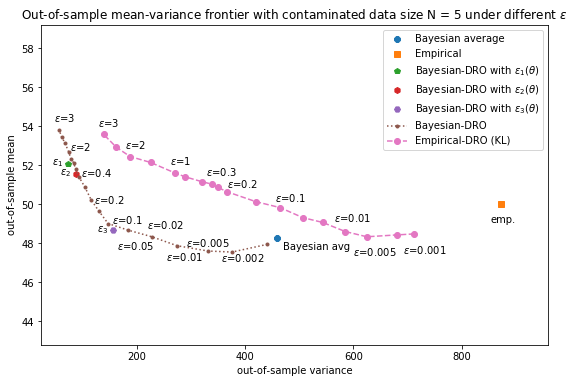}
        \caption{$N=5$.}
        \label{fig: newsvendor_m5_con}
    \end{subfigure}\hfill%
    \begin{subfigure}[t]{0.45\textwidth}
        \centering     \includegraphics[width=\textwidth]{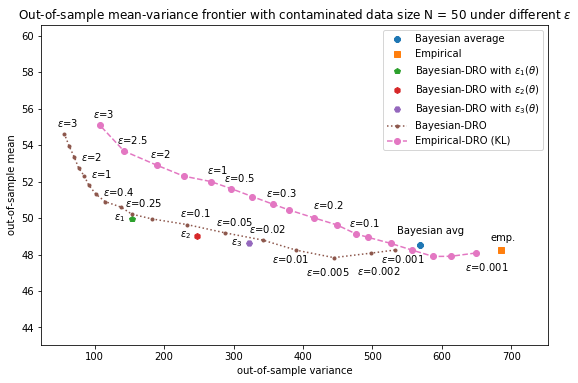}
        \caption{$N=50$.}
        \label{fig: newsvendor_m50_con}
    \end{subfigure}
    \caption{Newsvendor with finite support: out-of-sample mean-variance frontiers of different algorithms under different $\epsilon$ values with contaminated data. Data size $N$ is 5 and 50 respectively.}
    \label{fig: newsvendor_con}
\end{figure}

\begin{table}[h!]
\centering
{\small
\begin{tabular}{|c|c|c|c|c|c|c|}
\hline
N=5 & $\epsilon_1$ & $\epsilon_2$ & $\epsilon_3$ & Bayesian avg & empirical & true \\
\hline
$\epsilon$ value & 1.00({0.02}) & 0.50({0.01}) & 0.07({0.01}) & - & - & - \\
\hline
solution & 12.02({0.10}) & 11.72({0.13}) & 10.01({0.16}) & 8.59({0.08}) & 7.41({0.24}) & 7.00 \\
\hline
mean & 52.09({0.16}) & 51.58({0.27}) & 48.68({0.26}) & 48.29({0.05}) & 50.04({0.21}) & 47.21 \\
\hline
variance & \hspace{0.06cm} 73.12({1.97}) & \hspace{0.06cm} 87.09({1.98}) & 155.98({2.27}) & 459.39({2.47}) & 872.79({4.43}) & 770.56 \\
\hline
\end{tabular}
}
{\small
\begin{tabular}{|c|c|c|c|c|c|c|}
\hline
N=50 & $\epsilon_1$ & $\epsilon_2$ & $\epsilon_3$ & Bayesian avg & empirical & true \\
\hline
$\epsilon$ value & 0.13({0.00}) & 0.06({0.00}) & 0.02({0.00}) & - & - & - \\
\hline
solution & 11.32({0.05}) & 10.72({0.06}) & \hspace{0.06cm} 9.18({0.05}) & 7.86({0.11}) & 7.46({0.12}) & 7.00\\
\hline
mean & 49.97({0.06}) & 49.04({0.07}) & 48.63({0.11}) & 48.53({0.03}) & 48.26({0.04}) & 47.21 \\
\hline
variance & 153.08({1.41}) & 247.25({1.53}) & 321.62({1.63}) & 568.38({2.58}) & 684.59({2.69}) & 770.56 \\
\hline
\end{tabular}
}
\\
\caption{Newsvendor with finite support: out-of-sample performance of variants of Bayesian-DRO algorithm with contaminated data. Data size $N$ is 5 and 50 respectively.}
\label{Table: newsvendor_con}
\end{table}

\subsection{Multi-dimensional Portfolio Optimization with Finite-support Randomness}
In this subsection, we consider a five-dimensional portfolio optimization problem when the randomness $\xi$ has finite support and the data all come from the true distribution. We summarize notations used in the portfolio optimization problem as follows.
\begin{itemize}
    \item $x$: holding positions of assets. $x \in [0,1]^{5}$, $\sum_{i=1}^{5}x_i=1$.
    \item $\xi$: random returns of assets. $\xi_i$ takes values in $\{-1,0,1\}$ for $i=1,\cdots,5$.
\end{itemize}

The cost function is given by $G(x,\xi)=-\xi^{\top} x$. Note that we do not allow shorting (i.e., $x_i > 0, i=1,\cdots,5$) and impose a budget constraint ($\sum_{i=1}^{5} x_i=1$). The true probability mass of dimension $i$, denoted by $\theta^c_i \in \Delta_3$, is unknown to the decision maker.

\begin{figure}[!tbh]
    \centering
    \begin{subfigure}[t]{0.45\textwidth}
        \centering     \includegraphics[width=\textwidth]{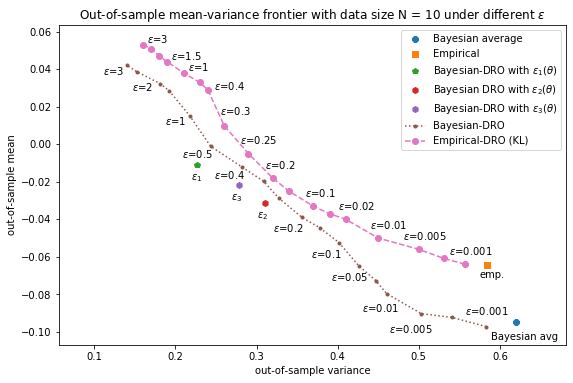}
        \caption{$N=10$.}
        \label{fig: portfolio_m10_non}
    \end{subfigure}\hfill%
    \begin{subfigure}[t]{0.45\textwidth}
        \centering     \includegraphics[width=\textwidth]{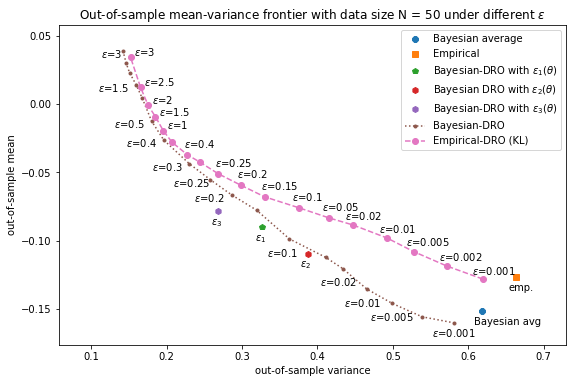}
        \caption{$N=50$.}
        \label{fig: portfolio_m50_non}
    \end{subfigure}
    \caption{Portfolio optimization with finite support: out-of-sample mean-variance frontiers of different algorithms under different $\epsilon$ values. Data size $N$ is 10 and 50 respectively.}
    \label{fig: portfolio_non}
\end{figure}

\begin{table}[htp]
\centering
{\small
\begin{tabular}{|c|c|c|c|c|c|c|}
\hline
N=10 & $\epsilon_1$ & $\epsilon_2$ & $\epsilon_3$ & Bayesian avg & empirical & true \\
\hline
$\epsilon$ value & 0.49({0.01}) & 0.25({0.01}) & 0.47({0.02}) & - & - & - \\
\hline
sol error & 0.73({0.01}) & 0.70({0.02}) & 0.79({0.01}) & 0.76({0.05}) & 0.80({0.04}) & 0.00 \\
\hline
mean & -0.01({0.00}) & -0.03({0.00}) & -0.02({0.00}) & -0.09({0.01}) & -0.06({0.01}) & -0.17 \\
\hline
variance & 0.23({0.01}) & 0.31({0.01}) & 0.28({0.00}) & 0.62({0.01}) & 0.58({0.02}) & 0.59 \\
\hline
\end{tabular}
}
{\small
\begin{tabular}{|c|c|c|c|c|c|c|}
\hline
N=50 & $\epsilon_1$ & $\epsilon_2$ & $\epsilon_3$ & Bayesian avg & empirical & true \\
\hline
$\epsilon$ value & 0.10({0.00}) & 0.05({0.00}) & 0.18({0.01}) & - & - & - \\
\hline
sol error & 0.53({0.02}) & 0.48({0.02}) & 0.61({0.01}) & 0.33({0.04}) & 0.40({0.04}) & 0.00 \\
\hline
mean & -0.09({0.00}) & -0.11({0.00}) & -0.08({0.00}) & -0.15({0.00}) & -0.13({0.00}) & -0.17 \\
\hline
variance & 0.33({0.01}) & 0.39({0.01}) & 0.27({0.00}) & 0.62({0.01}) & 0.66({0.01}) & 0.59\\
\hline
\end{tabular}
}
\\
\caption{Portfolio optimization with finite support: out-of-sample performance of variants of Bayesian-DRO. Data size $N$ is 10 and 50 respectively.}
\label{Table: portfolio_non}
\end{table}

Figure~\ref{fig: portfolio_non} shows the out-of-sample mean-variance frontiers (with varying $\epsilon$ values) of different algorithms for data size $N=5$ and $50$. Table~\ref{Table: portfolio_non} shows the out-of-sample performance of variants of Bayesian-DRO when data size is $N=10$ and 50 respectively. In addition to out-of-sample performance, we also show the solution error, which is obtained by calculating each solution's Euclidean distance from the true optimal solution, with sample standard deviation within the parentheses in the tables. Similar to the one-dimensional newsvendor problem, \textbf{Bayesian-DRO outperforms other benchmarks} in the multi-dimensional portfolio optimization problem.

\end{document}